\documentclass[11pt,a4paper]{article}

\usepackage{graphicx,latexsym,euscript,makeidx,color,bm}
\usepackage{amsmath,amsfonts,amssymb,amsthm,mathtools,mathrsfs,enumerate}
\usepackage[authoryear,round,sort]{natbib}
\setcitestyle{aysep={, },yysep={; }}
\usepackage[colorlinks,linkcolor=blue,anchorcolor=green,citecolor=black]{hyperref}

\usepackage{tikz}
\usetikzlibrary{fit,calc,positioning}
\tikzstyle{Block}=[rectangle,minimum width=3cm,minimum height=1cm,text centered,text width=5.2cm,draw=black]
\tikzstyle{Implication}=[rectangle,minimum width=3cm,minimum height=1cm,text centered,text width=5.2cm]
\tikzstyle{jian}=[<->, >=stealth]

\usepackage{geometry}
\allowdisplaybreaks[4]

      \def\hA{\widehat{A}} \def\tA{\widetilde{A}}
      \def\hB{\widehat{B}} 
\def\dbC{\mathbb{C}}  \def\cC{{\cal C}}    \def\hC{\widehat{C}} \def\tC{\widetilde{C}}
      
\def\dbE{\mathbb{E}}    
\def\dbF{\mathbb{F}}  \def\cF{{\cal F}}

  \def\cK{{\cal K}}  
  \def\cL{{\cal L}}  
     
\def\dbN{\mathbb{N}}  \def\cN{{\cal N}}  
    
  \def\cP{{\cal P}}  
    
\def\dbR{\mathbb{R}}  \def\cR{{\cal R}}  
\def\dbS{\mathbb{S}}  \def\cS{{\cal S}}  
  \def\cT{{\cal T}}

  \def\cW{{\cal W}}  
  \def\cX{{\cal X}}

\def\hb{\hbox}
\def\ms{\medskip}

        \def\lan{\langle}    \def\as{\hb{a.s.}}
   \def\ran{\rangle}    \def\tr{\hb{tr$\,$}}
         \def\rank{\hb{rank\,}}
\def\no{\noindent}

     \def\({\Big (}       \def\ba{\begin{aligned}}
      \def\){\Big )}       \def\ea{\end{aligned}}
      \def\[{\Big[}        \def\bel{\begin{equation}\label}
          \def\]{\Big]}        \def\ee{\end{equation}}

  \def\G{\Gamma}

\def\bp{\begin{pmatrix}}
\def\ep{\end{pmatrix}}

                 \def\hb{\hbox}
\def\ms{\medskip}              
          \def\lan{\langle}       \def\as{\text{a.s.}}
         \def\ran{\rangle}       \def\tr{\hb{tr$\,$}}
             
\def\no{\noindent}

  \def\({\Big(}           
       \def\){\Big)}           
   \def\[{\Big[}           
     \def\]{\Big]}

\DeclareMathOperator{\diag}{diag}

\DeclareMathOperator{\im}{im}
\DeclareMathOperator{\Span}{span}

\newtheoremstyle{thry}
{}      
{}      
{\sl}   
{}      
{\bf}   
{.}     
{.5em}  
{}      

\theoremstyle{thry}

\newtheorem{theorem}{Theorem}[section]
\newtheorem{proposition}[theorem]{Proposition}
\newtheorem{corollary}[theorem]{Corollary}
\newtheorem{lemma}[theorem]{Lemma}

\theoremstyle{definition}
\newtheorem{definition}[theorem]{Definition}

\newtheorem{assumption}[theorem]{Assumption}

\theoremstyle{remark}
\newtheorem{remark}[theorem]{Remark}

\makeatletter
   
   \@addtoreset{equation}{section}
   \newcommand{\setword}[2]{%
   \phantomsection
   #1\def\@currentlabel{\unexpanded{#1}}\label{#2}%
   }
\makeatother

\begin{document}

\title{\bf Hautus criteria for exact controllability and stabilizability of discrete-time backward-structured stochastic linear systems}
\author{Xurun Zuo\thanks{Correspongding author.
Department of Mathematics, Southern University of Science and Technology, Shenzhen 518055, China (Email: {\tt 12331006@mail.sustech.edu.cn}).} }
\date{}

\maketitle

\no{\bf Abstract.}
This paper studies exact controllability and $L^2$-stabilizability of discrete-time backward-structured stochastic linear systems.
For each prescribed finite horizon, equivalent characterizations of exact controllability are established in terms of the controllability Gramian, a Kalman-type rank condition, and the reachable subspace.
Exact null controllability without a prescribed horizon is then characterized by a finite-dimensional rank condition, and a Hautus-type criterion is derived using positive semidefinite eigenmatrices of a positive operator. 
For $L^2$-stabilizability, we establish a stabilizability decomposition and obtain a corresponding Hautus-type spectral criterion.
As a byproduct, we show that exact null controllability implies $L^2$-stabilizability.
These results provide algebraic and spectral tests for the controllability and stabilizability of the considered systems.

\ms
\no{\bf Key words.}
Exact controllability, exact null controllability, $L^2$-stabilizability, backward-structured stochastic difference equation, Hautus-type criteria.


\section{Introduction}

Let $(\Omega,\mathcal{F},\mathbb P)$ be a complete probability space carrying a sequence $\{w_k\}_{k\geq0}$ of independent scalar-valued random variables such that
$$
 \mathbb{E} w_k=0,\quad \mathbb{E} w_k^2=1,\quad k\in\dbN.
$$
Set $\mathcal{F}_k=\sigma(w_0,\ldots,w_k)$, augmented by the null sets, and let $\mathcal{F}_{-1}$ be the completed trivial sigma-field.  
Let $\dbF:=\{\cF_k\}_{k\geq-1}$.
Denote by $l_{\dbF}^2(0,\infty;\dbR^m)$ the space of all $\dbF$-predictable, square-summable $\dbR^m$-valued processes on $\{0,1,2,\cdots\}$. 
Consider the controlled discrete-time linear stochastic system
\begin{align}
\label{eq:system}
    \begin{cases}
        x_{k+1}=Ax_k+Bu_k+w_k(Cx_k+Du_k),\quad k\in\dbN,\\
        x_0=y\in\dbR^n,
    \end{cases}
\end{align}
where the coefficients $A,C\in\dbR^{n\times n}$ and $B,D\in\dbR^{n\times m}$ are constant matrices.
System \eqref{eq:system} is said to be $L^2$-stabilizable if, for any $y\in\dbR^n$, there exists a control process $u\in l_{\dbF}^2(0,\infty;\dbR^m)$ such that the corresponding state $x\in l_{\dbF}^2(0,\infty;\dbR^n)$.
This definition gives the open-loop formulation of $L^2$-stabilizability. 
The corresponding closed-loop formulation is introduced in \autoref{def:stabilizable}, and the equivalence between the two formulations is established in \autoref{lem:stabilizable}.

Feedback stabilizability of stochastic systems with state- and control-dependent noise was studied early by \citet{willems1976feedback}, and \citet{morozan1983stabilization} subsequently considered the stabilizability of stochastic discrete-time control systems. 
Since then, the stability and stabilizability of discrete-time stochastic systems have been investigated from several perspectives.
\citet{hou2011general}, \citet{sun2011critical}, and
\citet{hosoe2019equivalent} studied Lyapunov inequalities and related
stability and stabilization criteria, while a positive-operator approach was developed by \citet{ungureanu2013positive}.
Infinite-horizon stochastic linear-quadratic control and the associated Riccati equations were investigated by \citet{huang2008infinite}.
For discrete-time mean-field stochastic systems, \citet{ni2015discrete} established the equivalence between open-loop and closed-loop $L^2$-stabilizability. 
More recently, \citet{sun2025infinite} developed an infinite-horizon BSDE approach to exponential stabilizability and investigated its relationship with controllability. 
The stabilizability of continuous-time stochastic systems has also been extensively studied \citep[see, e.g.,][]{willems1976feedback,zhang2004stabilizability,zhang2008generalized,sun2026hautus}. 

Exact controllability of continuous-time stochastic systems was first studied by \citet{peng1994backward}, who showed that exact controllability requires the control coefficient in the diffusion term to have full row rank.
Motivated by this result, when $\rank D=n$, system \eqref{eq:system} can be equivalently rewritten as the following backward-structured stochastic system \citep{xu2023multiplicative}:
\begin{align}
\label{eq:canonical}
    \begin{cases}
        x_{k+1}=Ax_k+Bu_k+Cz_k+w_kz_k,\quad k\in\dbN,\\
        x_0=y\in\dbR^n.
    \end{cases}
\end{align}
\citet{wang2023exact} investigated the exact controllability of forward and backward stochastic difference systems with multiplicative noise.
\citet{xu2023multiplicative} established controllability criteria in terms of the controllability Gramian matrix and Kalman-type rank conditions. 
Related extensions address exact reachability under partial information \citep{wang2025partial} and exact controllability under control constraints \citep{xu2025constraint}.

The Popov--Belevitch--Hautus (PBH) test is one of the central spectral tools in deterministic linear control.  
For the discrete-time system $x_{k+1}=Ax_k+Bu_k$, controllability is equivalent to 
$$\rank(\lambda I-A,\,B)=n,\quad\forall\lambda\in\dbC,$$
while stabilizability is equivalent to the same rank condition restricted to $|\lambda|\geq1$ \citep[see, e.g.,][]{hespanha2018linear,trentelman2001control}. 
Stochastic counterparts of the PBH test for discrete-time systems have
been studied mainly in connection with observability and detectability
\citep[see, e.g.,][]{dragan2020exact,hou2016spectral,li2009detectability}.
Recently, \citet{sun2026hautus} established Hautus-type criteria for exact controllability and stabilizability of continuous-time backward-structured stochastic systems.
To the best of our knowledge, corresponding criteria for exact controllability and $L^2$-stabilizability of discrete-time backward-structured stochastic systems have not yet been developed.
Extending the classical PBH test to the present stochastic setting is not straightforward. 
In the presence of multiplicative noise, the reachable directions are generated jointly by the drift- and noise-related matrices, rather than by powers of a single system matrix.
Consequently, the classical left-eigenvector argument cannot be applied directly.
The relevant spectral obstructions are instead positive semidefinite eigenmatrices of a positive linear operator on $\dbS^n$. 

In this paper, we develop Hautus-type criteria for exact controllability and $L^2$-stabilizability of the backward-structured stochastic system \eqref{eq:canonical}.
Compared with the continuous-time results of \citet{sun2026hautus}, the discrete-time problem considered here presents several genuinely different features.
First, the corresponding backward stochastic difference equation does not necessarily admit an adapted solution for every square-integrable terminal state, which necessitates restricting the terminal state to set $\cS_N$ defined by \eqref{set:SN}.
 Second, exact controllability may depend on the prescribed horizon when $N<n-1$, although the reachable-subspace sequence stabilizes after at most $n-1$ steps. 
 Third, the $L^2$-stabilizability criterion tests positive semidefinite eigenmatrices associated with eigenvalues $\lambda\in[0,1]$, whereas the corresponding continuous-time criterion concerns eigenvalues $\lambda\leq0$.
The main contributions of this paper are summarized as follows.

\begin{enumerate}[(i)]
    \item 
    For each prescribed horizon $N$, we first introduce exact controllability on $[0,N]$, and prove that it is equivalent to exact null controllability and exact reachability from the zero initial state to every terminal state in the admissible terminal set $\cS_N$; see \autoref{prop:controllable-equivalent}.
   We then establish equivalent characterizations of these properties in terms of the controllability Gramian $G_N$, the Kalman-type matrix $R_N$, and the reachable subspace $V_N$; see \autoref{thm:controllable-N}.
    Since the sequence of subspaces $\{V_k\}_{k\geq0}$ stabilizes in at most $n-1$ steps, exact controllability on $[0,N]$ is characterized by the horizon-independent condition $\dim V_{n-1}=n$ whenever $N\geq n-1$.
    However, such horizon independence need not hold when $N<n-1$.
    \item 
    We next remove the prescribed-horizon requirement and define exact null controllability by requiring every initial state to be steered to zero in finitely many steps. Although the number of control steps may initially depend on the initial state, we prove that this definition is equivalent to exact null controllability on $[0,N]$ for some common finite horizon $N\in\dbN$; see \autoref{prop:null-controllable}. 
    Consequently, exact null controllability admits the finite-dimensional rank criterion $\rank R_{n-1}=\dim V_{n-1}=n$; see \autoref{thm:null-controllable}.
    We further establish a controllability decomposition associated with $V_{n-1}$, and then derive a Hautus-type criterion for exact null controllability of system \eqref{eq:canonical} using the spectral properties of positive operators; see \autoref{thm:controllability decomposition} and \autoref{thm:controllable-hautus}.

    \item 
    For $L^2$-stabilizability, we first consider the special case $B=0$ and obtain equivalent characterizations in terms of a strict matrix inequality and the spectral properties of a positive operator; see \autoref{prop:stabilizable-B=0}.
    We then prove that exact null controllability implies $L^2$-stabilizability and establish a stabilizability decomposition that reduces the $L^2$-stabilizability of system \eqref{eq:canonical} to that of the lower-dimensional subsystem \eqref{eq:subsystem-2}; see \autoref{lem:controllable-stabilizable} and \autoref{thm:stabilizable-decomposition}. Combining this decomposition with the spectral characterization for the case $B=0$, we derive a Hautus-type criterion for $L^2$-stabilizability of system \eqref{eq:canonical}; see \autoref{thm:stabilizable-hautus}.
\end{enumerate}

The rest of the paper is organized as follows.
\autoref{sec:preliminaries} collects the definitions of exact controllability and $L^2$-stabilizability, together with some preliminary results.
\autoref{sec:controllability} is devoted to exact controllability of system \eqref{eq:canonical}.
\autoref{sec:stabilizability} establishes the Hautus-type characterization of $L^2$-stabilizability.
Finally, \autoref{sec:conclusion} concludes the paper.

The following notations will be used throughout this paper.
$\dbR^{n\times m}$ denotes the space of real $n\times m$ matrices equipped with the Frobenius inner product $\lan\cdot,\cdot\ran$:
$$\lan M,N\ran\triangleq\tr(M^{\top}N),\quad M,N\in\dbR^{n\times m}.$$
The norm of $M\in\dbR^{n\times m}$ induced by $\lan\cdot,\cdot\ran$ is denoted by $|M|$.
For a matrix $M\in\dbR^{n\times m}$, we denote by $M^{\top}$ its transpose; for $M\in\dbR^{n\times n}$, $\tr(M)$ denotes its trace.
Let $\dbS^n\subset\dbR^{n\times n}$ denote the subspace of symmetric matrices. 
   We write $\mathbb S_+^n$ and $\bar{\mathbb S}_{+}^n$ for the sets of positive definite and positive semidefinite matrices in $\dbS^n$, respectively.
   For $M,N\in\dbS^n$, the notation $M\geq N$ (resp., $M>N$) means that $M-N$ is positive semidefinite (resp., positive definite).
   The identity matrix in $\dbR^{n\times n}$ is denoted by $I_n$, or simply by $I$ when no confusion can arise.
    Denote by $\im A$ the range of the matrix $A$.
    For $A\in\dbS^n$, $\lambda_{\min}(A)$ and $\lambda_{\max}(A)$ denote its minimum and maximum eigenvalues, respectively.
    For a subspace $V\subset\dbR^n$, denote by $\dim V$ the dimension of $V$.
    Denote by $\cL^*$ the adjoint of operator $\cL$ and by $r(\cL)$ the spectral radius of $\cL$.
   Denote by $\dbN$ the set $\{0,1,2,\cdots\}$ and $\dbN_+$ the set $\{1,2,\cdots\}$.
   For a random vector $\xi$, we write $\xi\in\cF_k$ to mean that $\xi$ is $\cF_k$-measurable.
   For a stochastic process $x$, we write $x\in\dbF$ if $x$ is predictable with respect to $\dbF$.
   For any $N\in\dbN$, denote $[0,N]=\{0,1,\cdots,N\}$.
   Finally, we introduce three spaces:
   \begin{align*}
       &L_{\cF_k}^2(\Omega;\dbR^n)\triangleq\{\xi:\Omega\to\dbR^n\ |\ \xi\in\cF_k\text{ and }\dbE|\xi|^2<\infty\},\\
       &l_{\dbF}^2(0,N;\dbR^n)\triangleq\left\{u:[0,N]\times\Omega\to\dbR^n\ |\ u\in\dbF\text{ and }\sum_{k=0}^N\dbE|u_k|^2<\infty\right\},\\
       &l_{\dbF}^2(0,\infty;\dbR^n)\triangleq\left\{u:\dbN\times\Omega\to\dbR^n\ |\ u\in\dbF\text{ and }\sum_{k=0}^{\infty}\dbE|u_k|^2<\infty\right\}.
   \end{align*}

\section{Preliminaries}
\label{sec:preliminaries}

In this section, we review the notions of controllability and stabilizability, and establish several preliminary results that will be used in the subsequent analysis.
   
We first introduce the notions of exact controllability and exact null controllability of backward-structured system \eqref{eq:canonical}.

\begin{definition}
\label{def:null-controllable-N}
    System \eqref{eq:canonical} is {\it exactly null controllable on $[0,N]$} if, for every $y\in\dbR^n$, there exist $(u,z)\in l_{\dbF}^2(0,N;\dbR^m)\times l_{\dbF}^2(0,N;\dbR^n)$ such that $x_{N+1}=0\ \as$
\end{definition}

\begin{remark}
\label{rk:controllable}
    The definition of exact null controllability for the above discrete-time system is analogous to that in continuous-time setting \citep{sun2026hautus}.
    However, exact controllability exhibits an essential difference between the discrete-time and continuous-time cases.
    This difference arises because the following terminal-value stochastic difference equation
    \begin{align}
    \label{eq:backward-example}
        \begin{cases}
            x_{k+1}=Ax_k+Bu_k+Cz_k+w_kz_k,\\
            x_{N+1}=\xi,
        \end{cases}
    \end{align}
    may not admit an adapted solution $(x,z)$ for every terminal state $\xi\in L_{\cF_N}^2(\Omega;\dbR^n)$, even when $u$ is allowed to vary over $l_{\dbF}^2(0,N;\dbR^m)$.
    For example, consider the special case in which $\{w_k\}_{k\geq0}$ is a sequence of standard normal random variables.
    Let $\xi=w_N^2e\in L_{\cF_N}^2(\Omega;\dbR^n)$,
    where $0\neq e\in\dbR^n$.
    Suppose that equation \eqref{eq:backward-example} admits a solution $(x,z)\in l_{\dbF}^2(0,N+1;\dbR^n)\times l_{\dbF}^2(0,N;\dbR^n)$.
    Then 
    \begin{align}
    \label{eq:remark}
        w_N^2e=Ax_N+Bu_N+Cz_N+w_Nz_N.
    \end{align}
    Taking conditional expectations in \eqref{eq:remark} with respect to $\cF_{N-1}$, we obtain
    \begin{align}
    \label{eq:remark-1}
        e=Ax_N+Bu_N+Cz_N.
    \end{align}
    Combining \eqref{eq:remark} and \eqref{eq:remark-1}, we have $(w_N^2-1)e=w_Nz_N$, which implies that $(w_N^3-w_N)e=w_N^2z_N$.
    Taking conditional expectations with respect to $\cF_{N-1}$ again, we have 
    \begin{align*}
        z_N=\dbE(w_N^2z_N|\cF_{N-1})=\dbE(w_N^3-w_N|\cF_{N-1})e=0.
    \end{align*}
    Thus, $(w_N^2-1)e=0$, which is impossible.
    This example shows that, unlike in the continuous-time setting, the terminal state space for exact controllability cannot in general be taken as the whole space $L_{\cF_N}^2(\Omega;\dbR^n)$; it must be restricted to those terminal states for which (2.1) admits an adapted solution.
\end{remark}

According to \autoref{rk:controllable}, we need to define an appropriate terminal state set such that system \eqref{eq:backward-example} is solvable.  
Inspired by \citet{xu2023multiplicative}, we consider the uncontrolled backward stochastic linear system
\begin{align}
\label{eq:BSDE-uncontrolled}
    \begin{cases}
        x_{k+1}=Ax_k+Cz_k+w_kz_k,\quad k\in[0,N],\\
        x_{N+1}=\xi.
    \end{cases}
\end{align}
Define the terminal-state set $\cS_N$ by
\begin{align}
\label{set:SN}
\begin{split}
    \cS_N:=\{\xi\in L_{\cF_N}^2(\Omega;\dbR^n)\ |\ &\text{system \eqref{eq:BSDE-uncontrolled} admits a solution}\\
     &(x,z)\in l_{\dbF}^2(0,N+1;\dbR^n)\times l_{\dbF}^2(0,N;\dbR^n)\}.
\end{split}
\end{align}
Clearly, for every $\xi\in\cS_N$, system \eqref{eq:backward-example} is solvable by taking $u=0$.
The following result provides an explicit representation of the elements of $\cS_N$ and shows that $\cS_N$ is a non-trivial subspace of $L_{\cF_N}^2(\Omega;\dbR^n)$.

\begin{lemma}
\label{lem:SN}
    For any $N\in\dbN$, we have
    \begin{align*}
        \cS_N=\{A^{N+1}y+\sum_{k=0}^NA^{N-k}(C+w_kI)z_k:(y,z)\in\dbR^n\times l_{\dbF}^2(0,N;\dbR^n)\}.
    \end{align*}
    Consequently, $\cS_N$ is a linear subspace of $L_{\cF_N}^2(\Omega;\dbR^n)$.
\end{lemma}

\begin{proof}
    Suppose that $\xi\in\cS_N$.
    Then there exists $(x,z)\in l_{\dbF}^2(0,N+1;\dbR^n)\times l_{\dbF}^2(0,N;\dbR^n)$ such that \eqref{eq:BSDE-uncontrolled} holds.
    Hence, we have
    \begin{align*}
        \xi&=x_{N+1}=Ax_N+(C+w_NI)z_N\\
        &=A[Ax_{N-1}+(C+w_{N-1}I)z_{N-1}]+(C+w_NI)z_N\\
        &=A^2x_{N-1}+A(C+w_{N-1}I)z_{N-1}+(C+w_NI)z_N\\
        &=\cdots\\
        &=A^{N+1}x_0+\sum_{k=0}^{N}A^k(C+w_{N-k}I)z_{N-k}\\
        &=A^{N+1}x_0+\sum_{k=0}^NA^{N-k}(C+w_kI)z_k.
    \end{align*}

    Conversely, fix any $(y,z)\in\dbR^n\times l_{\dbF}^2(0,N;\dbR^n)$, let $\xi=A^{N+1}y+\sum_{k=0}^NA^{N-k}(C+w_kI)z_k$.
    Define the process $x=(x_k)_{k\in[0,N+1]}$ recursively by
    \begin{align*}
        \begin{cases}
            x_0=y,\\
            x_{k+1}=Ax_k+(C+w_kI)z_k,\quad k\in[0,N].
        \end{cases}
    \end{align*}
    Then $x\in l_{\dbF}^2(0,N+1;\dbR^n)$ and $(x,z)$ satisfies equation \eqref{eq:BSDE-uncontrolled}.
    The proof is done.
\end{proof}

Based on the set $\cS_N$, we can define the exact controllability of system \eqref{eq:canonical}.

\begin{definition}
    System \eqref{eq:canonical} is {\it exactly controllable on $[0,N]$} if, for every $(y,\xi)\in\dbR^n\times\cS_N$, there exists $(u,z)\in l_{\dbF}^2(0,N;\dbR^m)\times l_{\dbF}^2(0,N;\dbR^n)$ such that $x_{N+1}=\xi\ \as$
\end{definition}

The following result establishes the equivalence between exact controllability and exact null controllability of system \eqref{eq:canonical} on $[0,N]$. 
Moreover, both are equivalent to the reachability of every element of $\cS_N$ from the zero initial state.

\begin{proposition}
\label{prop:controllable-equivalent}
    The following statements are equivalent.
    \begin{enumerate}[(i)]
        \item System \eqref{eq:canonical} is exactly controllable on $[0,N]$.
        \item System \eqref{eq:canonical} is exactly null controllable on $[0,N]$.
        \item For any $\xi\in\cS_N$, there exists $(u,z)\in l_{\dbF}^2(0,N;\dbR^m)\times l_{\dbF}^2(0,N;\dbR^n)$ such that
        \begin{align*}
            \begin{cases}
                x_{k+1}=Ax_k+Bu_k+Cz_k+w_kz_k,\quad k\in[0,N],\\
                x_0=0,\quad x_{N+1}=\xi.
            \end{cases}
        \end{align*}
    \end{enumerate}
\end{proposition}

\begin{proof}
    (i) $\implies$ (ii).
    This follows directly from $0\in\cS_N$.

    (ii) $\implies$ (i).
    Fix $(y,\xi)\in\dbR^n\times\cS_N$.
    By the definition of $\cS_N$, there exists $(\bar x,\bar z)\in l_{\dbF}^2(0,N+1;\dbR^n)\times l_{\dbF}^2(0,N;\dbR^n)$ such that 
    \begin{align*}
    \begin{cases}
        \bar x_{k+1}=A\bar x_k+C\bar z_k+w_k\bar z_k,\quad k\in[0,N],\\
        \bar x_{N+1}=\xi.
    \end{cases}
    \end{align*}
    On the other hand, by assumption, there exists $(\hat u,\hat z)\in l_{\dbF}^2(0,N;\dbR^m)\times l_{\dbF}^2(0,N;\dbR^n)$ such that 
    \begin{align*}
        \begin{cases}
            \hat x_{k+1}=A\hat x_k+B\hat u_k+C\hat z_k+w_k\hat z_k,\quad k\in[0,N],\\
            \hat x_0=y-\bar x_0,\quad\hat x_{N+1}=0.
        \end{cases}
    \end{align*}
    Let $x_k=\bar x_k+\hat x_k$, $u_k=\hat u_k$, $z_k=\bar z_k+\hat z_k$.
    Then 
    \begin{align}
    \label{eq:prop-equivalence-controllable}
        \begin{cases}
            x_{k+1}=Ax_k+Bu_k+Cz_k+w_kz_k,\quad k\in[0,N],\\
            x_0=y,\quad x_{N+1}=\xi.
        \end{cases}
    \end{align}
    Hence, system \eqref{eq:canonical} is exactly controllable on $[0,N]$.

    (i) $\implies$ (iii). 
    This follows by taking $y=0$ in (i).

    (iii) $\implies$ (i).
    Fix $(y,\xi)\in\dbR^n\times\cS_N$.
    By \autoref{lem:SN}, $A^{N+1}y\in\cS_N$ and then $\xi-A^{N+1}y\in\cS_N$.
    By assumption, there exist $\bar x\in l_{\dbF}^2(0,N+1;\dbR^n)$ and  $(u,z)\in l_{\dbF}^2(0,N;\dbR^m)\times l_{\dbF}^2(0,N;\dbR^n)$ such that
    \begin{align*}
        \begin{cases}
            \bar x_{k+1}=A\bar x_k+Bu_k+Cz_k+w_kz_k,\quad k\in[0,N],\\
            \bar x_0=0,\quad \bar x_{N+1}=\xi-A^{N+1}y.
        \end{cases}
    \end{align*}
    Let $x_k=\bar x_k+A^{k}y$ for any $k\in[0,N+1]$.
    Then $(x,u,z)$ satisfies system \eqref{eq:prop-equivalence-controllable}.
    The proof is done.
\end{proof}

The invertibility of $A$ is necessary for the unique solvability of the terminal-value stochastic system \eqref{eq:backward-example}.
Indeed, suppose that $A$ is not invertible.
Set $\xi=0$ and $u=0$. 
Define $x\in l_{\cF}^2(0,N+1;\dbR^n)$
by $x_0=v$ and $x_k=0$ for $k\in[1,N+1]$.
Let $z=0$.
Then $(x,z)$
satisfies \eqref{eq:backward-example} for any $v\in\ker A$, which implies that \eqref{eq:backward-example} is not uniquely solvable.
This motivates the following assumption.
\begin{assumption}
\label{ass:Ainvertible}
The matrix $A\in\mathbb{R}^{n\times n}$ is invertible.
\end{assumption}

\begin{remark}
    Consider the continuous controlled system 
    \begin{align*}
        dX_t=(AX_t+Bu_t+Cz_t)dt+z_tdW_t,
    \end{align*}
    where $\{W_t\}_{t\geq0}$ is a one-dimensional standard Brownian motion.
    Applying the Euler-Maruyama scheme with a uniform step size $h\  (h>0)$ gives
    \begin{align*}
        X_{t_{k+1}}=(I+hA)X_{t_k}+hBu_{t_k}+hCz_{t_k}+\sqrt{h}w_kz_{t_k},
    \end{align*}
    where $w_k=\frac{W_{t_{k+1}}-W_{t_k}}{\sqrt{h}}$ is a standard normal random variable.
    For sufficiently small $h$, the matrix $I+hA$ is invertible. Therefore, \autoref{ass:Ainvertible} is naturally satisfied by discrete-time models arising from sufficiently fine time discretizations of continuous-time stochastic systems.
\end{remark}

Before introducing the notion of stabilizability, we recall the concept of stability.
Although our main focus is on discrete-time backward-structured stochastic systems, we consider the following general linear system in order to develop the results in a more general framework.
Consider the system
\begin{align}
\label{eq:system-uncontrolled}
    \begin{cases}
        x_{k+1}=Ax_k+Cw_kx_k,\quad k\in\dbN,\\
        x_0=y.
    \end{cases}
\end{align}

\begin{definition}
\label{def:stable}
    \begin{enumerate}[(i)]
        \item System \eqref{eq:system-uncontrolled} is said to be {\it $L^2$-stable} if 
        \begin{align*}
            \sum_{k=0}^{\infty}\dbE|x_k|^2<\infty,\quad\forall y\in\dbR^n.
        \end{align*}
        \item System \eqref{eq:system-uncontrolled} is said to be {\it $L^2$-asymptotically stable} if 
    \begin{align*}
        \lim_{k\to\infty}\dbE|x_k|^2=0,\quad\forall y\in\dbR^n.
    \end{align*}
    \item System \eqref{eq:system-uncontrolled} is said to be {\it $L^2$-exponentially stable} if there exists constant $c>0$ and $\rho\in(0,1)$ such that 
    \begin{align*}
        \dbE|x_k|^2\leq c\rho^k|y|^2,\quad\forall k\in\dbN,\quad\forall y\in\dbR^n.
    \end{align*}
    \end{enumerate}
\end{definition}

The following lemma establishes the equivalence among the three notions of stability in \autoref{def:stable} and provides several useful equivalent characterizations. 
Although these results have been discussed in the literature \citep[see, e.g.,][]{ni2015discrete,hosoe2019equivalent}, we include a brief proof for completeness.

\begin{lemma}
\label{lem:stable}
    The following statements are equivalent.
    \begin{enumerate}[(i)]
        \item System \eqref{eq:system-uncontrolled} is $L^2$-stable.
        \item System \eqref{eq:system-uncontrolled} is $L^2$-asymptotically stable.
        \item System \eqref{eq:system-uncontrolled} is $L^2$-exponentially stable.
        \item There exists $P\in\dbS_+^n$ such that 
        \begin{align*}
            P-A^{\top}PA-C^{\top}PC>0.
        \end{align*}
        \item For any $Q\in\dbS_+^n$, the algebraic equation
        \begin{align*}
            P-A^{\top}PA-C^{\top}PC=Q
        \end{align*}
        admits a unique solution $P\in\dbS_+^n$.
        \item $r(\cT)<1$, where $\cT:\dbS^n\to\dbS^n$ is defined by 
        \begin{align*}
            \cT(P)=A^{\top}PA+C^{\top}PC,\quad\forall P\in\dbS^n.
        \end{align*}
    \end{enumerate}
\end{lemma}

\begin{proof}
    Note that for any $P\in\bar\dbS_+^n$ and $k\in\dbN$, 
    \begin{align}
    \label{eq:one-step-second-moment}
        \dbE(x_{k+1}^{\top}Px_{k+1}|\cF_{k-1})=x_k^{\top}(A^{\top}PA+C^{\top}PC)x_k=x_k^{\top}\cT(P)x_k.
    \end{align}
    Iterating this identity gives 
    \begin{align}
    \label{eq:second-moment}
        \dbE(x^{\top}_kPx_k)=y^{\top}\cT^k(P)y,\quad\forall P\in\bar\dbS_+^n,\quad\forall k\in\dbN.
    \end{align} 

    (ii) $\implies$ (vi).
    By \eqref{eq:second-moment} and (ii), we have $\lim_{k\to\infty}y^{\top}\cT^k(I)y=0$ for any $y\in\dbR^n$, which implies that $\lim_{k\to\infty}\cT^k(I)=0$.
    For any $H\in\dbS^n$, choose $\alpha>0$ such that $-\alpha I\leq H\leq\alpha I$.
    Then $-\alpha\cT^k(I)\leq\cT^k(H)\leq\alpha\cT^k(I)$.
    Therefore, $\lim_{k\to\infty}\cT^k(H)=0$ for every $H\in\dbS^n$.
    Since $\cT$ is a finite-dimensional operator, this implies that $\lim_{k\to\infty}\|\cT^k\|=0$.
    Hence, every eigenvalue $\lambda$ of the complexification of $\cT$ satisfies $|\lambda|<1$, and consequently $r(\cT)<1$.

    (vi) $\implies$ (iii).
    Choose $\rho\in(0,1)$ such that $r(\cT)<\rho<1$.
    Then, there exists a constant $c>0$ such that 
    \begin{align*}
        \dbE|x_k|^2=y^{\top}\cT^k(I)y\leq c\rho^k|y|^2,\quad\forall k\in\dbN.
    \end{align*}
    
    (iii) $\implies$ (i) $\implies$ (ii) is obvious.

    (vi) $\implies$ (v).
    Since $r(\cT)<1$, $P:=\sum_{j=0}^{\infty}\cT^j(Q)$ is well-defined for any $Q\in\dbS_+^n$.
    Furthermore, $P\geq Q>0$ and $\cT(P)=\sum_{j=1}^{\infty}\cT^j(Q)$.
    Therefore, $P-\cT(P)=Q$.
    The uniqueness of $P$ follows from the invertibility of $I-\cT$.

    (v) $\implies$ (iv).
    It follows immediately by taking $Q=I$.

    (iv) $\implies$ (i).
    By (iv), set $Q:=P-\cT(P)>0$ with $P>0$.
    By \eqref{eq:one-step-second-moment}, we obtain
    \begin{align*}
        \dbE(x_{k+1}^{\top}Px_{k+1})-\dbE(x_k^{\top}Px_k)=-\dbE(x_k^{\top}Qx_k)\leq-\lambda_{\min}(Q)\dbE|x_k|^2,\quad\forall k\in\dbN.
    \end{align*}
    Summing the above inequality from $k=0$ to $k=N$ gives
    \begin{align*}
        \lambda_{\min}(Q)\sum_{k=0}^N\dbE|x_k|^2\leq y^{\top}Py-\dbE(x_{N+1}^{\top}Px_{N+1})\leq\lambda_{\max}(P)|y|^2.
    \end{align*}
    Letting $N\to\infty$, we obtain (i).
    The proof is done.
\end{proof}

\begin{definition}
\label{def:stabilizable}
    System \eqref{eq:system} is {\it $L^2$-stabilizable} if there exists $\Theta\in\dbR^{m\times n}$ such that system 
    \begin{align}
    \label{eq:closed-loop}
        \begin{cases}
            x_{k+1}=(A+B\Theta)x_k+(C+D\Theta)w_kx_k,\\
            x_0=y,
        \end{cases}
    \end{align}
    is $L^2$-stable.
    Any such matrix $\Theta$ is called a {\it stabilizer} of system \eqref{eq:system}.
\end{definition}

We end this section by presenting several equivalent characterizations of $L^2$-stabilizability.

\begin{lemma}
\label{lem:stabilizable}
    The following statements are equivalent.
    \begin{enumerate}[(i)]
        \item System \eqref{eq:system} is $L^2$-stabilizable.
        \item For any $y\in\dbR^n$, there exists $u\in l_{\dbF}^2(0,\infty;\dbR^m)$ such that the corresponding state process $x\in l_{\dbF}^2(0,\infty;\dbR^n)$.
        \item The algebraic Riccati equation 
        \begin{align}
        \label{eq:algebraic-Riccati}
            \begin{split}
                P=&I+A^{\top}PA+C^{\top}PC\\
            &\quad-(A^{\top}PB+C^{\top}PD)(I+B^{\top}PB+D^{\top}PD)^{-1}(B^{\top}PA+D^{\top}PC)
            \end{split}
        \end{align}
        admits a solution $P\in\dbS_+^n$.
        In this case, 
        \begin{align}
        \label{eq:stabilizer}
            \Theta=-(I+B^{\top}PB+D^{\top}PD)^{-1}(B^{\top}PA+D^{\top}PC)
        \end{align}
        is a stabilizer of system \eqref{eq:system}.
    \end{enumerate}
\end{lemma}

\begin{proof}
    (i) $\implies$ (ii).
    Suppose that $\Theta$ is a stabilizer of system \eqref{eq:system}.
    Let $u=\Theta\bar x$, where $\bar x\in l_{\cF}^2(0,\infty;\dbR^n)$ is the solution of \eqref{eq:closed-loop}.
    Substituting $u$ into \eqref{eq:system}, the solution $x$ coincides with $\bar x$.
    Therefore, $x\in l_{\dbF}^2(0,\infty;\dbR^n)$.

    (ii) $\implies$ (iii).
    Let $e_1,\cdots,e_n$ be the standard basis of $\dbR^n$.
    Denote by $x(y,u)$ the solution of \eqref{eq:system} with initial state $y$ and control input $u$.
    For each $i=1,\cdots,n$, there exists $u^i\in l_{\dbF}^2(0,\infty;\dbR^m)$ such that $x^i\equiv x(e_i,u_i)\in l_{\dbF}^2(0,\infty;\dbR^n)$.
    Hence,
    \begin{align*}
        J^i:=\sum_{k=0}^{\infty}\dbE(|x_k^i|^2+|u_k^i|^2)<\infty.
    \end{align*}
    Suppose that $y=\sum_{i=1}^na_ie_i$, where $a_i\in\dbR$.
    Let $u^y=\sum_{i=1}^na_iu_i$.
    Then $x^y\equiv x(y,u^y)=\sum_{i=1}^na_ix^i\in l_{\dbF}^2(0,\infty;\dbR^n)$ by linearity of \eqref{eq:system}.
    Moreover, 
    \begin{align}
    \label{eq:second-moment-upper-bound}
    \begin{split}
        &\sum_{k=0}^{\infty}\dbE(|x_k^y|^2+|u_k^y|^2)=\sum_{k=0}^{\infty}\dbE\left(\left|\sum_{i=1}^na_ix_k^i\right|^2+\left|\sum_{i=1}^na_iu_k^i\right|^2\right)\\
        \leq&\sum_{k=0}^{\infty}\dbE(|y^2|\sum_{i=1}^n|x_k^i|^2+|y|^2\sum_{i=1}^n|u_k^i|^2)=|y|^2\sum_{i=1}^n\sum_{k=0}^{\infty}\dbE(|x_k^i|^2+|u_k^i|^2)\\
        =&|y|^2\sum_{i=1}^nJ^i:=c|y|^2.
    \end{split}
    \end{align}
    For any fixed $N\in\dbN_+$, it follows from the classical LQ problem \citep[see, e.g.,][]{rami2002discrete,elliott2013discrete} that the value function associated with the following optimal control problem
    \begin{align*}
        V_N(y)\triangleq\inf_{u\in l_{\cF}^2(0,N-1;\dbR^m)}\dbE\sum_{k=0}^{N-1}(|x_k|^2+|u_k|^2)
    \end{align*}
    is given by $V_N(y)=y^{\top}P_Ny$, where $\{P_k\}_{k\geq0}$ is generated by the difference Riccati equation
    \begin{align}
    \label{eq:difference-Riccati}
        \begin{cases}
            P_{k+1}=I+A^{\top}P_kA+C^{\top}P_kC\\
            \qquad\qquad-(A^{\top}P_kB+C^{\top}P_kD)(I+B^{\top}P_kB+D^{\top}P_kD)^{-1}(B^{\top}P_kA+D^{\top}P_kC),\\
            P_0=0.
        \end{cases}
    \end{align}
    Note that $V_{N+1}(y)\geq V_{N}(y)$ for any $y\in\dbR^n$.
    We have $P_{N+1}\geq P_N$ for any $N\in\dbN$.
    On the other hand, $y^{\top}P_Ny\leq c|y|^2$ by \eqref{eq:second-moment-upper-bound}.
    Then $P_N\leq cI$ for any $N\in\dbN$.
    Letting $k\to\infty$ in \eqref{eq:difference-Riccati}, we obtain that $\cP\triangleq\lim_{k\to\infty}P_k$ satisfies \eqref{eq:algebraic-Riccati}.

    (iii) $\implies$ (i).
    Suppose that $P\in\dbS_+^n$ is a solution of \eqref{eq:algebraic-Riccati}.
    Let $\Theta$ be defined by \eqref{eq:stabilizer}.
    Then 
    \begin{align*}
        (A+B\Theta)^{\top}P(A+B\Theta)+(C+D\Theta)^{\top}P(C+D\Theta)-P=-I-\Theta^{\top}\Theta<0.
    \end{align*}
    By \autoref{lem:stable} and \autoref{def:stabilizable}, system \eqref{eq:system} is $L^2$-stabilizable.
    The proof is done.
\end{proof}

\section{Exact controllability}
\label{sec:controllability}

In this section, we investigate the exact controllability of the backward-structured system \eqref{eq:canonical}. 
We first consider controllability over a prescribed finite horizon and derive equivalent characterizations in terms of the controllability Gramian, a Kalman-type rank condition, and the reachable subspace. 
We then remove the prescribed-horizon requirement, establish a controllability decomposition, and derive a Hautus-type criterion for exact null controllability.

By \autoref{ass:Ainvertible}, system \eqref{eq:canonical} can be rewritten as 
\begin{align}
    \label{eq:backward}
    \begin{cases}
        x_k=\widehat{A} x_{k+1}+\widehat{B} u_k+\widehat{C} z_k-\widehat{A} w_kz_k,\quad k\in\dbN,\\
        x_0=y\in\dbR^n,
    \end{cases}
\end{align}
where 
\begin{equation}
\label{eq:hA-A}
 \widehat{A}=A^{-1},\qquad \widehat{B}=-A^{-1}B,\qquad \widehat{C}=-A^{-1}C.
\end{equation}
We first define a matrix-valued stochastic process $\{\Phi_k\}_{k\geq0}$ by
\begin{align}
\label{eq:Phi}
    \begin{cases}
        \Phi_0=I,\\
        \Phi_k=M_0M_1\cdots M_{k-1},\quad k\in\dbN+,
    \end{cases}
\end{align}
where $M_k=\hA+w_k\hC$ for any $k\in\dbN$.
By the definition of $\Phi$, we have the following result.

\begin{lemma}
\label{lem:BSDE-express}
    Suppose that \autoref{ass:Ainvertible} holds.
    Let $u\in l_{\dbF}^2(0,N;\dbR^m)$.
    Suppose that equation \eqref{eq:backward} with terminal state $x_{N+1}=\xi\in L_{\cF_N}^2(\Omega;\dbR^n)$ admits a solution $(x,z)\in l_{\dbF}^2(0,N+1;\dbR^n)\times l_{\dbF}^2(0,N;\dbR^n)$.
    Then, for every $k\in[0,N]$, the solution satisfies 
    \begin{align}
    \label{eq:z-representation}
        &z_k=\dbE(w_kx_{k+1}|\cF_{k-1}),\quad\forall k\in[0,N],\\
    \label{eq:x-representation}
        &x_k=\dbE(M_kx_{k+1}+\hB u_k|\cF_{k-1}),\quad\forall k\in[0,N].
    \end{align}
    Moreover, the initial state can be represented by 
    \begin{align*}
        x_0=\dbE(\Phi_{N+1}\xi)+\dbE\sum_{k=0}^N\Phi_k\hB u_k,
    \end{align*}
    where $\{\Phi_k\}_{k\geq0}$ is defined by \eqref{eq:Phi}.
\end{lemma}

\begin{proof}
    For any $k\in[0,N]$, taking conditional expectation on both sides of equation \eqref{eq:backward} with respect to $\cF_{k-1}$, we have 
    \begin{align}
    \label{eq:x-equality}
        x_k&=\dbE(x_k|\cF_{k-1})=\hA\dbE(x_{k+1}|\cF_{k-1})+\hB u_k+\hC z_k.
    \end{align}
    Combining \eqref{eq:x-equality} and \eqref{eq:backward}, we have
    \begin{align}
    \label{eq:z-equality}
        w_kz_k=x_{k+1}-\dbE(x_{k+1}|\cF_{k-1}),\quad\forall k\in[0,N].
    \end{align}
    Multiplying both sides of equation \eqref{eq:z-equality} by $w_k$ and then taking the conditional expectation with respect to $\cF_{k-1}$ yields \eqref{eq:z-representation}.
    Substituting \eqref{eq:z-representation} into \eqref{eq:x-equality}, we obtain \eqref{eq:x-representation}.
    Consequently, we have 
    \begin{align*}
        x_0&=\dbE(M_0x_1+\hB u_0|\cF_{-1})=\dbE(\Phi_1x_1+\Phi_0\hB u_0)\\
        &=\dbE[\Phi_1\dbE(M_1x_2+\hB u_1|\cF_0)+\Phi_0\hB u_0]\\
        &=\dbE(\Phi_2x_2+\Phi_1\hB u_1+\Phi_0\hB u_0)\\
        &=\cdots\\
        &=\dbE(\Phi_{N+1}x_{N+1})+\dbE\sum_{k=0}^N\Phi_k\hB u_k.
    \end{align*}
    The proof is done.
\end{proof}

By \autoref{lem:BSDE-express}, the contribution of the control sequence to the initial state is determined by the random matrices $\{\Phi_k\hB\}_{k\geq0}$. This motivates us to introduce the following controllability Gramian, which will be used to characterize the controllability of system \eqref{eq:canonical}.
For any $N\in\dbN$, define the controllability Gramian 
\begin{align}
\label{eq:Gramian}
    G_N:=\sum_{k=0}^N\dbE(\Phi_k\hB\hB^{\top}\Phi_k^{\top}).
\end{align}
To obtain a recursive representation of $G_N$, define the positive linear operator $\cL:\dbS^n\to\dbS^n$ by 
\begin{align}
\label{eq:cL}
    \cL(P)=\hA P\hA^{\top}+\hC P\hC^{\top},\quad\forall P\in\dbS^n.
\end{align}
Hereafter, unless otherwise specified, $\cL$ denotes the operator defined in \eqref{eq:cL}.

\begin{remark}
    With respect to the Frobenius inner product, the adjoint operator of $\cL$ is given by
    \begin{align*}
        \cL^*(P)=\hA^{\top}P\hA+\hC^{\top}P\hC,\quad\forall P\in\dbS^n.
    \end{align*}
    Hence, both $\cL$ and $\cL^*$ are positive linear operators.
    If $0\neq H\in\bar\dbS_+^n$ is an eigenmatrix of either $\cL$ or $\cL^*$, then its corresponding eigenvalue $\lambda$ is nonnegative.
    Indeed, $\lambda H=\cL(H)\geq0$ or $\lambda H=\cL^*(H)\geq0$, which implies that $\lambda\geq0$.
\end{remark}

The relationship between $G_N$ and $\cL$ is characterized by the following lemma.

\begin{lemma}
\label{lem:GN}
    For any $k\in\dbN$, we have $\cL^k(\hB\hB^{\top})=\dbE(\Phi_k\hB\hB^{\top}\Phi_k^{\top})$.
    Consequently, 
    \begin{align}
        G_N=\sum_{k=0}^N\cL^k(\hB\hB^{\top})=\hB\hB^{\top}+\cL(G_{N-1}),\quad\forall N\in\dbN_+.
    \end{align}
\end{lemma}

\begin{proof}
    For any $k\in\dbN$ and $P\in\dbS^n$, we have
    \begin{align*}
    \dbE(M_kPM_k^{\top})&=\dbE[(\hA+w_k\hC)P(\hA+w_k\hC)^{\top}]\\
    &=\dbE(\hA P\hA^{\top}+w_k\hA P\hC^{\top}+w_k\hC P\hA^{\top}+w_k^2\hC P\hC^{\top})\\
    &=\hA P\hA^{\top}+\hC P\hC^{\top}=\cL(P).
    \end{align*}
    Then, for any $k\in\dbN$, using the independence of $\{M_k\}_{k\geq0}$, we obtain
    \begin{align*}
        \dbE(\Phi_k\hB\hB^{\top}\Phi_k^{\top})&=\dbE(M_0\cdots M_{k-1}\hB\hB^{\top}M_{k-1}^{\top}\cdots M_0^{\top})\\
        &=\dbE[M_0\dbE(M_1\cdots M_{k-1}\hB\hB^{\top}M_{k-1}^{\top}\cdots M_1^{\top})M_0^{\top}]\\
        &=\cL(\dbE(M_1\cdots M_{k-1}\hB\hB^{\top}M_{k-1}^{\top}\cdots M_1^{\top}))\\
        &=\cdots\\
        &=\cL^k(\hB\hB^{\top}).
    \end{align*}
    The proof is done.
\end{proof}

To derive a finite-dimensional rank characterization of $G_N$, we introduce the following recursively defined matrices.
Define 
\begin{align}
\label{eq:R}
    \begin{cases}
        R_0=\hB,\\
        R_{k+1}=(\hB,\hA R_k,\hC R_k),\quad k\in\dbN.
    \end{cases}
\end{align}

\begin{lemma}
\label{lem:G-R}
    For any $k\in\dbN$, $G_k=R_kR_k^{\top}$.
\end{lemma}

\begin{proof}
    For $k=0$, $G_0=\hB\hB^{\top}=R_0R_0^{\top}$.
    Suppose that $G_k=R_kR_k^{\top}$ for some $k\in\dbN$.
    By \autoref{lem:GN}, we have
    \begin{align*}
        G_{k+1}&=\hB\hB^{\top}+\cL(G_k)=\hB\hB^{\top}+\cL(R_kR_k^{\top})\\
        &=\hB\hB^{\top}+\hA R_kR_k^{\top}\hA^{\top}+\hC R_kR_k^{\top}\hC^{\top}\\
        &=(\hB,\hA R_k,\hC R_k)(\hB,\hA R_k,\hC R_k)^{\top}\\
        &=R_{k+1}R_{k+1}^{\top}.
    \end{align*}
    The proof is done.
\end{proof}

By \autoref{lem:G-R}, $G_N$ is invertible if and only if $R_N$ has full row rank.
To describe the column spaces of $R_k$ recursively, we introduce the following sequence of subspaces:
\begin{align}
\label{eq:V}
    \begin{cases}
        V_0=\im \hB,\\
        V_{k+1}=V_0+\hA V_k+\hC V_k,\quad k\in\dbN.
    \end{cases}
\end{align}

\begin{lemma}
\label{lem:V-R}
    For any $k\in\dbN$, $V_k=\im R_k$.
\end{lemma}

\begin{proof}
    Note that $V_0=\im R_0$.
    Suppose that $V_k=\im R_k$ for some $k\in\dbN$.
Then 
\begin{align*}
    \im R_{k+1}&=\im(\hB,\hA R_k,\hC R_k)=\im\hB+\hA\im R_k+\hC\im R_k\\
    &=V_0+\hA V_k+\hC V_k=V_{k+1}.
\end{align*}
By induction, we finish the proof.
\end{proof}

The preceding lemmas relate the controllability Gramian $G_N$ to $R_N$ and $V_N$. These relations lead to the following finite-horizon controllability criterion.

\begin{theorem}
\label{thm:controllable-N}
    Suppose that \autoref{ass:Ainvertible} holds.
    Let $N\in\dbN$.
    Then the following statements are equivalent.
    \begin{enumerate}[(i)]
        \item System \eqref{eq:canonical} is exactly controllable on $[0,N]$.
        \item System \eqref{eq:canonical} is exactly null controllable on $[0,N]$.
        \item $G_N$ is invertible, where $G_N$ is defined by \eqref{eq:Gramian}.
        \item $\rank R_N=n$, where $R_N$ is defined by \eqref{eq:R}.
        \item $\dim V_N=n$, where $V_N$ is defined by \eqref{eq:V}.
    \end{enumerate}
\end{theorem}

\begin{proof}
    (i) $\iff$ (ii) follows from \autoref{prop:controllable-equivalent}.

    (ii) $\implies$ (iii).
    Suppose that $G_N$ is not invertible.
    Then there exists $0\neq\eta\in\dbR^n$ such that $G_N\eta=0$.
    Hence, we have
    \begin{align*}
        \sum_{k=0}^N\dbE(\eta^{\top}\Phi_k\hB\hB^{\top}\Phi_k^{\top}\eta)=\eta^{\top}G_N\eta=0,
    \end{align*}
    which implies that $\eta^{\top}\Phi_k\hB=0\ \as$ for any $k\in[0,N]$.
    On the other hand, there exists $(u,z)\in l_{\dbF}^2(0,N;\dbR^m)\times l_{\dbF}^2(0,N;\dbR^n)$ such that $x_0=\eta,x_{N+1}=0$, where $x$ satisfies equation \eqref{eq:backward}.
    By \autoref{lem:BSDE-express}, $\eta=\dbE\sum_{k=0}^N\Phi_k\hB u_k$.
    Hence, 
    \begin{align*}
        \eta^{\top}\eta=\dbE\sum_{k=0}^N\eta^{\top}\Phi_k\hB u_k=0,
    \end{align*}
    which contradicts $\eta\neq0$.

    (iii) $\implies$ (ii).
    Fix any $y\in\dbR^n$ and set $g=G_N^{-1}y$.
    Let 
    \begin{align}
    \label{eq:construction}
        \begin{cases}
            x_k=G_{N-k}\Phi_k^{\top}g,\quad k\in[0,N],\\
            u_k=\hB^{\top}\Phi_k^{\top}g,\quad k\in[0,N],\\
            z_k=G_{N-k-1}\hC^{\top}\Phi_k^{\top}g,\quad k\in[0,N-1],\\
            z_N=0,\\
            x_{N+1}=0.
        \end{cases}
    \end{align}
    Then $(u,z)\in l_{\dbF}^2(0,N;\dbR^m)\times l_{\dbF}^2(0,N;\dbR^n)$.
    Moreover, for any $k\in[0,N-1]$, we have 
    \begin{align*}
        &\hA x_{k+1}+\hB u_k+\hC z_k-\hA w_kz_k\\
        =&\hA G_{N-k-1}\Phi_{k+1}^{\top}g+\hB\hB^{\top}\Phi_k^{\top}g+\hC G_{N-k-1}\hC^{\top}\Phi_k^{\top}g-\hA w_kG_{N-k-1}\hC^{\top}\Phi_k^{\top}g\\
        =&\hA G_{N-k-1}(\hA^{\top}+w_k\hC^{\top})\Phi_k^{\top}g+\hB\hB^{\top}\Phi_k^{\top}g+\hC G_{N-k-1}\hC^{\top}\Phi_k^{\top}g-\hA w_kG_{N-k-1}\hC^{\top}\Phi_k^{\top}g\\
        =&[\hB\hB^{\top}+\cL(G_{N-k-1})]\Phi_k^{\top}g\\
        =&G_{N-k}\Phi_k^{\top}g\\
        =&x_k.
    \end{align*}
    It remains to verify the initial state condition and the state equation at $k=N$.
    Note that
    \begin{align*}
        &x_0=G_N\Phi_0^{\top}g=G_Ng=y,\\
        &Ax_N+Bu_N+Cz_N+w_Nz_N=(A\hB\hB^{\top}+B\hB^{\top})\Phi_N^{\top}g=0=x_{N+1}.
    \end{align*}
    This proves that system \eqref{eq:canonical} is exactly null controllable on $[0,N]$.

    (iii) $\iff$ (iv) follows from \autoref{lem:G-R} and (iv) $\iff$ (v) follows from \autoref{lem:V-R}.
    The proof is done.
\end{proof}

\begin{remark}
    \citet{xu2023multiplicative} establish an equivalent characterization of controllability at some finite horizon. 
    In contrast, \autoref{thm:controllable-N} gives equivalent characterizations of controllability over each prescribed horizon $N$, thereby providing a more refined result.
\end{remark}

\begin{corollary}
    Suppose that \autoref{ass:Ainvertible} holds.
    Suppose that system \eqref{eq:canonical} is exactly null controllable on $[0,N]$.
    Then system \eqref{eq:canonical} is exactly (null) controllable on $[0,N+k]$ for any $k\in\dbN$.
\end{corollary}

\begin{proof}
    By \autoref{thm:controllable-N}, $G_N$ is invertible.
    Note that $G_{N+k}\geq G_N$ for any $k\in\dbN$.
    Then $G_{N+k}$ is invertible and we finish the proof by \autoref{thm:controllable-N} again.
\end{proof}

\autoref{def:null-controllable-N} requires exact null controllability over a prescribed finite horizon. The following definition relaxes this requirement by only requiring that the state be driven to zero in some finite number of steps.

\begin{definition}
\label{def:null-controllable}
    System \eqref{eq:canonical} is {\it exactly null controllable} if, for every $y\in\dbR^n$, there exist $N=N(y)\in\dbN$ and $(u,z)\in l_{\dbF}^2(0,N;\dbR^m)\times l_{\dbF}^2(0,N;\dbR^n)$ such that $x_{N+1}=0\ \as$
\end{definition}

\begin{proposition}
\label{prop:null-controllable}
The following statements are equivalent.
    \begin{enumerate}[(i)]
        \item System \eqref{eq:canonical} is exactly null controllable.
        \item There exists $N\in\dbN$ such that system \eqref{eq:canonical} is exactly (null) controllable on $[0,N]$.
    \end{enumerate}
\end{proposition}

\begin{proof}
    Define 
    \begin{align}
    \label{eq:cK_N}
    \begin{split}
        \cK_N=\{y\in\dbR^n\ |\ &\text{there exists }(u,z)\in l_{\dbF}^2(0,N;\dbR^m)\times l_{\dbF}^2(0,N;\dbR^n)\\
        &\text{ such that the solution } x\text{ of \eqref{eq:canonical} satisfies } x_{N+1}=0\}.
    \end{split}
    \end{align}
    Then $\cK_N$ is a nonempty subspace of $\dbR^n$.
    Suppose that $y\in\cK_N$.
    Then there exists $(u,z)\in l_{\dbF}^2(0,N;\dbR^m)\times l_{\dbF}^2(0,N;\dbR^n)$ such that $x_{N+1}=0$.
    Let
    \begin{align*}
        \bar u_k=
        \begin{cases}
            u_k,\quad0\leq k\leq N,\\
            0,\quad k=N+1,
        \end{cases}
        \quad
        \bar z_k=
        \begin{cases}
            z_k,\quad0\leq k\leq N,\\
            0,\quad k=N+1.
        \end{cases}
    \end{align*}
    Then $(\bar u,\bar  z)\in l_{\dbF}^2(0,N+1;\dbR^m)\times l_{\dbF}^2(0,N+1;\dbR^n)$ and the corresponding state $\bar x$ satisfies $\bar x_{N+2}=0$.
    Hence, $y\in\cK_{N+1}$ and then $\cK_N\subset\cK_{N+1}$.
    
    (i) $\implies$ (ii).
    Let $e_1,\cdots,e_n$ be the standard basis of $\dbR^n$.
    Then there exist $N_1,\cdots,N_n\in\dbN$ such that $e_i\in\cK_{N_i}$ for any $i\in[1,n]$.
    Define $N_*=\max_{1\leq i\leq n}N_i$.
    Then $e_i\in\cK_{N_*}$ for any $i\in[1,n]$.
    Hence, $\cK_{N_*}=\dbR^n$ and then system \eqref{eq:canonical} is exactly null controllable on $[0,N_*]$.

    (ii) $\implies$ (i).
    Obviously.
    The proof is done.
\end{proof}

\begin{remark}
    Under \autoref{ass:Ainvertible}, for every $N\in\dbN$, the subspace $\cK_N$ introduced in the proof of \autoref{prop:null-controllable} coincides with the subspace $V_N$ defined by \eqref{eq:V}.
    Indeed, let $y\in\cK_N$.
    Then there exists a pair $(u,z)\in l_{\dbF}^2(0,N;\dbR^m)\times l_{\dbF}^2(0,N;\dbR^n)$ such that the solution $x$ of \eqref{eq:canonical} satisfies $x_0=y$ and $x_{N+1}=0$.
    By \autoref{lem:BSDE-express}, $y=\dbE\sum_{k=0}^N\Phi_k\hB u_k$.
    For any $\eta\in\ker G_N$, we have
    \begin{align*}
        0=\eta^{\top}G_N\eta=\sum_{k=0}^N\dbE|\hB^{\top}\Phi_k^{\top}\eta|^2.
    \end{align*}
    Hence, $\hB^{\top}\Phi_k^{\top}\eta=0\ \as$ for every $k\in[0,N]$, and therefore
    \begin{align*}
        \eta^{\top}y=\dbE\sum_{k=0}^N\eta^{\top}\Phi_k\hB u_k=0.
    \end{align*}
    Since $G_N$ is symmetric, it follows that $\cK_N\subset(\ker G_N)^{\perp}=\im G_N$.
    Conversely, let $y\in\im G_N$.
    Then there exists $g\in\dbR^n$ such that $y=G_Ng$.
    Using this $g$ in the construction \eqref{eq:construction}, the same computation as in the proof of \autoref{thm:controllable-N} shows that the resulting $(x,u,z)$ satisfies \eqref{eq:canonical}, with $x_0=y$ and $x_{N+1}=0$.
    Hence, $y\in\cK_N$, and therefore $K_N=\im G_N$.
    By \autoref{lem:G-R} and \autoref{lem:V-R}, 
    \begin{align*}
        \cK_N=\im G_N=\im R_N=V_N.
    \end{align*}
\end{remark}

\autoref{prop:null-controllable} reduces exact null controllability to exact null controllability over some finite horizon.
By \autoref{thm:controllable-N}, the latter is equivalent to $\dim V_N=n$ for some $N\in\mathbb N$.
Since the horizon $N$ is not prescribed in \autoref{def:null-controllable}, we seek a criterion independent of horizon.
For this purpose, we first investigate the monotonicity and finite-step stabilization of the sequence $\{V_k\}_{k\geq0}$.
The following lemma collects the required properties.

\begin{lemma}
\label{lem:V}
    Suppose $\{V_k\}_{k\geq0}$ is defined by \eqref{eq:V}.
    Then the following statements hold.
    \begin{enumerate}[(i)]
        \item For any $k\in\dbN_+$, $V_{k-1}\subset V_k$.
        \item Suppose that $V_k=V_{k+1}$.
        Then $V_k=V_{k+l}$ for any $l\in\dbN$.
        \item There exists $0\leq k\leq n-1$ such that $V_k=V_{k+l}$ for any $l\in\dbN$.
        \item $V_{n-1}=V_{n+l}$ for any $l\in\dbN$.
        Consequently, $V_{n-1}$ is $\hA$- and $\hC$-invariant.
    \end{enumerate}
\end{lemma}

\begin{proof}
    (i). 
    Note that $V_0\subset V_1$.
    Suppose that $V_{k-1}\subset V_k$.
    Then 
    \begin{align*}
    V_k=V_0+\hA V_{k-1}+\hC V_{k-1}\subset V_0+\hA V_k+\hC V_k=V_{k+1}.
\end{align*}

(ii). 
Note that 
\begin{align*}
    V_{k+2}=V_0+\hA V_{k+1}+\hC V_{k+1}=V_0+\hA V_k+\hC V_k=V_{k+1}.
\end{align*}
Hence, we have $V_k=V_{k+1}=V_{k+2}=\cdots$.

(iii).
If $\hB=0$, then $\dim V_k=0$ for any $k\in\dbN$.
Now assume $\hB\neq0$.
Suppose that $1\leq\dim V_0<\dim V_1<\cdots<\dim V_{n-1}<\dim V_n$.
Then $\dim V_n\geq n+1$, a contradiction.
Hence, there exists $k\in[0,n-1]$ such that $V_k=V_{k+1}$.
By (ii), $V_k=V_{k+l}$ for every $l\in\dbN$.

(iv).
$V_{n-1}=V_{n+l}$ for any $l\in\dbN$ follows from (ii) and (iii) directly.
By $\hA V_{n-1}\subset V_n=V_{n-1}$ and $\hC V_{n-1}\subset V_n=V_{n-1}$, $V_{n-1}$ is $\hA$- and $\hC$-invariant.
The proof is done.
\end{proof}

\begin{remark}
    By \autoref{thm:controllable-N} and \autoref{lem:V}, for any $N\geq n-1$, system \eqref{eq:canonical} is exactly controllable on $[0,N]$ if and only if $\dim V_{n-1}=n$.
    Thus, once $N\geq n-1$, the exact controllability criterion is independent of the horizon.
    However, this conclusion does not hold when $N<n-1$.
    Indeed, let 
    \begin{align*}
        A=
        \begin{pmatrix}
            0&1\\1&0
        \end{pmatrix},\quad
        B=
        \begin{pmatrix}
            0\\1
        \end{pmatrix},\quad
        C=
        \begin{pmatrix}
            0&0\\0&0
        \end{pmatrix}.
    \end{align*}
    Then $\hA=A,\hB=(-1,0)^{\top},\hC=0$.
    Consequently, $\rank R_0=\rank\hB=1$, whereas $\rank R_1=\rank(\hB,\hA\hB,\hC\hB)=2$.
    Therefore, by \autoref{thm:controllable-N}, system \eqref{eq:canonical} is not exactly controllable on $[0,0]$, but it is exactly controllable on $[0,1]$.
    This example also shows that the bound $n-1$ in \autoref{lem:V} (iv) is sharp.
\end{remark}

\begin{theorem}
\label{thm:null-controllable}
    Suppose that \autoref{ass:Ainvertible} holds.
    Then the following statements are equivalent.
    \begin{enumerate}[(i)]
        \item System \eqref{eq:canonical} is exactly null controllable.
        \item There exists $N\in\dbN$ such that system \eqref{eq:canonical} is exactly (null) controllable on $[0,N]$.
        \item $\rank R_{n-1}=\dim V_{n-1}=n$.
    \end{enumerate}
\end{theorem}

\begin{proof}
    (i) $\iff$ (ii) follows from \autoref{prop:null-controllable}.

    (ii) $\implies$ (iii).
    By \autoref{thm:controllable-N}, $\dim V_N=n$.
    If $N\leq n-1$, then $V_N\subset V_{n-1}=\dbR^n$ by \autoref{lem:V} (i).
    If $N\geq n$, then $V_{n-1}=V_N=\dbR^n$ by \autoref{lem:V} (iv).
    Thus, (iii) holds by \autoref{lem:V-R}.

    (iii) $\implies$ (ii).
    By \autoref{thm:controllable-N}, system \eqref{eq:canonical} is exactly (null) controllable on $[0,n-1]$.
    The proof is done.
\end{proof}

To prepare for the controllability decomposition, we first give the notion of matrix words generated by matrices $A$ and $C$.
For an integer $k\geq0$, a word of length $k$ over $\{A,C\}$ is a matrix of the form
\begin{align*}
    \cX=X_1X_2\cdots X_k,\quad X_i\in\{A,C\},\quad i=1,\cdots,k.
\end{align*}
When $k=0$, the empty product is understood as the identity matrix $I$.
Let $\cW_k(A,C)$ denote the set of all words of length at most $k$ over $\{A,C\}$.

\begin{lemma}
\label{lem:R-word}
    $\im R_k=\Span\{\im(\cX\hB):\cX\in\cW_k(\hA,\hC)\}$.
\end{lemma}

\begin{proof}
    For $k=0$, $\Span\{\im\hB\}=\im\hB=\im R_0$.
    Suppose that $\im R_k=\Span\{\im(\cX\hB):\cX\in\cW_k(\hA,\hC)\}$.
    Then  
    \begin{align*}
        &\Span\{\im(\cX\hB):\cX\in\cW_{k+1}(\hA,\hC)\}\\
        =&\Span\{\im\hB,\im(\hA\cX\hB),\im(\hC\cX\hB):\cX\in\cW_k(\hA,\hC)\}\\
        =&\Span\{\im\hB,\hA\im(\cX\hB),\hC\im(\cX\hB):\cX\in\cW_k(\hA,\hC)\}\\
        =&\Span\{\im\hB,\hA\im R_k,\hC\im R_k\}\\
        =&\im R_{k+1}.
    \end{align*}
    By induction, we finish the proof.
\end{proof}

Based on this representation and the invariance of $V_{n-1}$ under $\hA$ and $\hC$, we obtain the following controllability decomposition.

\begin{theorem}[Controllability decomposition]
\label{thm:controllability decomposition}
    Suppose that \autoref{ass:Ainvertible} holds and $0<k\triangleq\dim V_{n-1}<n$.
    Then the following results hold.
    \begin{enumerate}[(i)]
        \item There exists an orthogonal matrix $P\in\dbR^{n\times n}$ such that 
    \begin{align}
    \label{eq:decomposition}
        P^{\top}\hA P=
        \begin{pmatrix}
            \hA_{11}&\hA_{12}\\
            0&\hA_{22}
        \end{pmatrix},
        \quad
        P^{\top}\hC P=
        \begin{pmatrix}
            \hC_{11}&\hC_{12}\\
            0&\hC_{22}
        \end{pmatrix},
        \quad
        P^{\top}\hB=
        \begin{pmatrix}
            \hB_1\\0
        \end{pmatrix},
    \end{align}
    where $\hA_{11},\hC_{11}\in\dbR^{k\times k}$, $\hA_{12},\hC_{12}\in\dbR^{k\times(n-k)}$, $\hA_{22},\hC_{22}\in\dbR^{(n-k)\times(n-k)}$ and $\hB_1\in\dbR^{k\times m}$.
    Moreover,
    \begin{align}
    \label{eq:decompostion-original-matrix}
        P^{\top}AP=
        \begin{pmatrix}
            A_{11}&A_{12}\\
            0&A_{22}
        \end{pmatrix},\quad
        P^{\top}CP=
        \begin{pmatrix}
            C_{11}&C_{12}\\
            0&C_{22}
        \end{pmatrix},\quad
        P^{\top}B=
        \begin{pmatrix}
            B_1\\0
        \end{pmatrix},
    \end{align}
    where 
    \begin{align}
    \label{eq:A11-hA11}
        \begin{cases}
            A_{11}=\hA_{11}^{-1},\ A_{12}=-\hA_{11}^{-1}\hA_{12}\hA_{22}^{-1},\ A_{22}=\hA_{22}^{-1},\\
            C_{11}=-A_{11}\hC_{11},\ C_{12}=-A_{11}\hC_{12}-A_{12}\hC_{22},\ C_{22}=-A_{22}\hC_{22},\\
            B_1=-A_{11}\hB_1
        \end{cases}
    \end{align}
    \item Let $S_0=\hB_1$ and $S_{j+1}=(\hB_1,\hA_{11}S_j,\hC_{11}S_j)$ for $j\in\dbN$.
    Then $\rank S_{k-1}=k$.
    \item The system 
    \begin{align}
    \label{eq:subsystem-1}
        \begin{cases}
            y_{1,k+1}=A_{11}y_{1,k}+B_1u_k+C_{11}v_{1,k}+w_kv_{1,k},\\
            y_{1,0}\in\dbR^k.
        \end{cases}
    \end{align}
    is exactly null controllable.
    \end{enumerate}
\end{theorem}

\begin{proof}
    Choose an orthonormal basis $\{e_1,\cdots,e_k\}$ of $V_{n-1}$ and extend it to an orthonormal basis $\{e_1,\cdots,e_n\}$ of $\dbR^n$.
    Let $P\triangleq(e_1,\cdots,e_n)$.
    By \autoref{lem:V}, $V_{n-1}$ is $\hA$- and $\hC$-invariant.
    Then
    \begin{align*}
        P^{\top}\hA P=
        \begin{pmatrix}
            \hA_{11}&\hA_{12}\\
            0&\hA_{22}
        \end{pmatrix},
        \quad
        P^{\top}\hC P=
        \begin{pmatrix}
            \hC_{11}&\hC_{12}\\
            0&\hC_{22}
        \end{pmatrix},
    \end{align*}
    where $\hA_{11},\hC_{11}\in\dbR^{k\times k}$.
    Also, since $\im\hB\subset V_{n-1}$, $P^{\top}\hB=(\hB_1^{\top},0)^{\top}$
    for some $\hB_1\in\dbR^{k\times m}$.
    Hence,
    \begin{align*}
            &P^{\top}AP=(P^{\top}\hA P)^{-1}=\begin{pmatrix}
            \hA_{11}&\hA_{12}\\
            0&\hA_{22}
        \end{pmatrix}^{-1}
        =
        \begin{pmatrix}
            \hA_{11}^{-1}&-\hA_{11}^{-1}\hA_{12}\hA_{22}^{-1}\\
            0&\hA_{22}^{-1}
        \end{pmatrix},\\
        &P^{\top}CP=-P^{\top}APP^{\top}\hC P=-
        \begin{pmatrix}
            A_{11}&A_{12}\\
            0&A_{22}
        \end{pmatrix}
        \begin{pmatrix}
            \hC_{11}&\hC_{12}\\
            0&\hC_{22}
        \end{pmatrix}
        =-
        \begin{pmatrix}
            A_{11}\hC_{11}&A_{11}\hC_{12}+A_{12}\hC_{22}\\
            0&A_{22}\hC_{22}
        \end{pmatrix},\\
        &P^{\top}B=-P^{\top}APP^{\top}\hB=-
        \begin{pmatrix}
            A_{11}&A_{12}\\
            0&A_{22}
        \end{pmatrix}
        \begin{pmatrix}
            \hB_1\\0
        \end{pmatrix}
        =-
        \begin{pmatrix}
          A_{11}\hB_1\\0  
        \end{pmatrix}
    \end{align*}
    This proves (i).
    Moreover, by \autoref{lem:V-R}, \autoref{lem:R-word}, and \autoref{lem:V},
    \begin{align*}
        k&=\dim V_{n-1}=\dim(P^{\top}V_{n-1})=\dim(P^{\top}\im R_{n-1})\\
        &=\dim\Span\{\im(P^{\top}\cX\hB):\cX\in\cW_{n-1}(\hA,\hC)\}\\
        &=\dim\Span\{\im(P^{\top}\cX PP^{\top}\hB):\cX\in\cW_{n-1}(\hA,\hC)\}\\
        &=\dim\Span\{\im(\cX P^{\top}\hB):\cX\in\cW_{n-1}(P^{\top}\hA P,P^{\top}\hC P)\}\\
        &=\dim\Span\{\im(\cX \hB_1):\cX\in \cW_{n-1}(\hA_{11},\hC_{11})\}\\
        &=\dim\im S_{n-1}=\dim\im S_{k-1}.
    \end{align*}
    Hence, (ii) holds.
    Consequently, (iii) holds by \eqref{eq:A11-hA11} and \autoref{thm:null-controllable}.
    The proof is done.
\end{proof}

To derive a Hautus-type controllability criterion, we use the following finite-dimensional version of the Krein–Rutman theorem on $\dbS^n$ (see \citealp[Theorem 3.1]{vandergraft1968spectral}).
\begin{lemma}
\label{lem:K-R}
    Let $K$ be a solid, closed, pointed, and convex cone in $\dbS^n$.
    Suppose that $\G:\dbS^n\to\dbS^n$ is a linear operator satisfying $\G(K)\subset K$.
    Then there exists $0\neq P\in K$ such that $\G(P)=r(\G)P$, where $r(\G)$ is the spectral radius of $\G$.
\end{lemma}

\begin{theorem}
\label{thm:controllable-hautus}
    Under \autoref{ass:Ainvertible}, the following statements are equivalent.
    \begin{enumerate}[(i)]
        \item System \eqref{eq:canonical} is exactly null controllable.
        \item For any eigenpair $(\lambda,H)\in[0,\infty)\times\bar\dbS_+^n\setminus\{0\}$ of $\cL^*$, we have $\hB^{\top}H\neq0$.
    \end{enumerate}
\end{theorem}

\begin{proof}
    (i) $\implies$ (ii).
    Suppose that system \eqref{eq:canonical} is exactly null controllable.
    By \autoref{thm:null-controllable} and \autoref{thm:controllable-N}, $G_N>0$ for some $N\in\dbN$.
    Let $\cL^*(H)=\lambda H$ with $(\lambda,H)\in[0,\infty)\times\bar\dbS_+^n\setminus\{0\}$.
    Then, by \autoref{lem:GN},
    \begin{align*}
        \lan G_N,H\ran=\left\lan\sum_{k=0}^N\cL^k(\hB\hB^{\top}),H\right\ran=\sum_{k=0}^N\lan\hB\hB^{\top},(\cL^*)^k(H)\ran=\sum_{k=0}^N\lambda^k\lan\hB^{\top},\hB^{\top}H\ran.
    \end{align*}
    If $\hB^{\top}H=0$, then 
    \begin{align*}
        0=\lan G_N,H\ran=\tr(G_NH)\geq\lambda_{\min}(G_N)\tr(H)>0,
    \end{align*}
    which is impossible.
    Hence, $\hB^{\top}H\neq0$.

    (ii) $\implies$ (i).
    Applying \autoref{lem:K-R} to $K=\bar\dbS_+^n$ and $\G=\cL^*$, and then using (ii), we obtain $\hB\neq0$.
    Suppose that system \eqref{eq:canonical} is not exactly null controllable.
    Then $0<k\triangleq\dim V_{n-1}<n$ by \autoref{thm:null-controllable}.
    Hence, we can use \autoref{thm:controllability decomposition}.
    Let $P$ be the orthogonal matrix provided by \autoref{thm:controllability decomposition} such that \eqref{eq:decomposition} holds.
    Define $\cL_2:\dbS^{n-k}\to\dbS^{n-k}$ by $\cL_2(H_2)=\hA_{22}H_2\hA_{22}^{\top}+\hC_{22}H_2\hC_{22}^{\top}$ for any $H_2\in\dbS_+^{n-k}$.
    By \autoref{lem:K-R}, there exists $(\lambda,H_2)\in[0,\infty)\times\bar\dbS_+^{n-k}\setminus\{0\}$ such that $\cL_2^*(H_2)=\lambda H_2$.
    Set $H=P\diag(0,H_2)P^{\top}$.
    Then $H\in\bar\dbS_+^n\setminus\{0\}$ and 
    \begin{align*}
        \cL^*(H)&=\hA^{\top}H\hA+\hC^{\top}H\hC\\
        &=PP^{\top}\hA^{\top}P\diag(0,H_2)P^{\top}\hA PP^{\top}+PP^{\top}\hC^{\top}P\diag(0,H_2)P^{\top}\hC PP^{\top}\\
        &=P\diag(0,\hA_{22}^{\top}H_2\hA_{22})P^{\top}+P\diag(0,\hC_{22}^{\top}H_2\hC_{22})P^{\top}\\
        &=P\diag(0,\cL_2^*(H_2))P^{\top}\\
        &=\lambda H.
    \end{align*}
    On the other hand,
    \begin{align*}
        \hB^{\top}HP=\hB^{\top}PP^{\top}HP=(\hB_1^{\top},0)\diag(0,H_2)=0.
    \end{align*}
    Hence, $\hB^{\top}H=0$, contradicting (ii).
    Therefore, system \eqref{eq:canonical} is exactly null controllable.
\end{proof}

\section{\texorpdfstring{$L^2$-stabilizability}{L2-stabilizability}}
\label{sec:stabilizability}

In this section, we investigate a Hautus-type criterion for the $L^2$-stabilizability of system \eqref{eq:canonical}. 
We first establish a spectral characterization for strict inequalities involving positive operators and apply it to the special case $B=0$.
We then show that exact null controllability implies $L^2$-stabilizability and establish a stabilizability decomposition that reduces the general problem to a lower-dimensional subsystem. 
Combining these results yields the desired Hautus-type criterion for system \eqref{eq:canonical}.

We begin with several elementary matrix inequalities involving positive (semi)definite matrices.
\begin{lemma}
\label{lem:matrix}
    \begin{enumerate}[(i)]
        \item Let $P,Q\in\bar\dbS_+^n$.
        Then $\lambda_{\min}(Q)\tr(P)\leq\tr(PQ)\leq\lambda_{\max}(Q)\tr(P)$.
        \item Let $P,Q\in\dbS_+^n$ and $A\in\dbR^{n\times n}$.
        Then $A^{\top}Q^{-1}A<P^{-1}$ if and only if $APA^{\top}<Q$.
        \item Let $P\in\dbS^n$.
        Then $P\geq0$ if and only if $\lan P,H\ran\geq0,\forall H\in\bar\dbS_+^n$.
    \end{enumerate}
\end{lemma}

\begin{proof}
\begin{enumerate}[(i)]
    \item 
    Note that $\lambda_{\min}(Q)I\leq Q\leq \lambda_{\max}(Q)I$.
    Then $\lambda_{\min}(Q)P
\leq P^{\frac{1}{2}}QP^{\frac{1}{2}}
\leq \lambda_{\max}(Q)P$, which implies that 
$$
\lambda_{\min}(Q)\operatorname{tr}(P)
\leq\tr(P^{\frac{1}{2}}QP^{\frac{1}{2}})= \operatorname{tr}(PQ)
\leq \lambda_{\max}(Q)\operatorname{tr}(P).
$$
\item 
Since $P>0$ and $Q>0$, the Schur complement yields
$$
A^\top Q^{-1}A<P^{-1}
\quad\Longleftrightarrow\quad
\begin{pmatrix}
P^{-1} & A^\top\\
A & Q
\end{pmatrix}>0
\quad\Longleftrightarrow\quad
Q-APA^\top>0.
$$
\item 
Suppose first that $P\geq0$. For every $H\in\bar\dbS_+^n$,
$$
\langle P,H\rangle
=\operatorname{tr}(PH)
=\operatorname{tr}\bigl(P^{1/2}HP^{1/2}\bigr)
\geq0.
$$
Conversely, suppose that$\langle P,H\rangle\geq0,\forall H\in\bar\dbS_+^n$.
For any $v\in\mathbb R^n$, take $H=vv^\top\in\bar\dbS_+^n$. Then
$$
v^\top Pv
=\operatorname{tr}(Pvv^\top)
=\langle P,vv^\top\rangle
\geq0.
$$
Hence $P\geq0$. The proof is done.
\end{enumerate}
\end{proof}

Based on \autoref{lem:matrix}, we next give several characterizations of the existence of a positive definite solution to the following strict operator inequality $\G(P)-P>0$, whose direction is opposite to that of the inequality in \autoref{lem:stable} (iv).

\begin{proposition}
\label{prop:positive-operator}
    Let $\G:\dbS^n\to\dbS^n$ be a positive operator.
    Then the following statements are equivalent.
    \begin{enumerate}[(i)]
        \item There exists $P\in\dbS_+^n$ such that $\G(P)-P>0$.
        \item If $S\in\bar\dbS_+^n$ and $\G^*(S)-S\leq0$, then $S=0$.
        \item If $H\in\bar\dbS_+^n\setminus\{0\}$ and $\G^*(H)=\lambda H$, then $\lambda>1$.
    \end{enumerate}
\end{proposition}

\begin{proof}
    (i) $\implies$ (ii).
 Suppose that $P>0$ and $\G(P)-P>0$.
 Let $S\in\bar\dbS_+^n$ and $\G^*(S)-S\leq0$.
    If $S\neq0$, then, by \autoref{lem:matrix} (i), we have 
    \begin{align*}
        0&<\lambda_{\min}(\G(P)-P)\tr(S)\leq\lan\G(P)-P,S\ran=\lan P,(\G^*-I)(S)\ran\\
        &=\tr(P(\G^*-I)(S))\leq\lambda_{\max}((\G^*-I)(S))\tr(P)\leq0,
    \end{align*}
   which gives a contradiction.
   Hence, $S=0$.

   (ii) $\implies$ (i).
   Argue by contradiction.
   Suppose that (i) fails.
   Define $\cC=\{\G(P)-P:P\in\dbS_+^n\}$.
   Then $\cC\cap\dbS_+^n=\emptyset$.
   Note that $\cC$ is a convex set and $\dbS_+^n$ is an open convex set.
   By Hahn-Banach theorem, there exist a nonzero linear functional $f$ on $\dbS^n$ and $\alpha\in\dbR$ such that 
   \begin{align*}
       f(P_1)\leq\alpha\leq f(P_2),\quad\forall P_1\in\cC,\quad\forall P_2\in\dbS_+^n.
   \end{align*}
   Note that for any $t>0,P_1\in\cC,P_2\in\dbS_+^n$, we have $tP_1\in\cC,tP_2\in\dbS_+^n$.
   Then 
   \begin{align*}
       tf(P_1)\leq\alpha\leq tf(P_2),\quad\forall P_1\in\cC,\quad\forall P_2\in\dbS_+^n,\quad\forall t>0.
   \end{align*}
   Let $t\to0$.
   Then $\alpha=0$.
   Hence, $f((\G-I)(P_1))\leq0$ for any $P_1>0$ and $f(P_2)\geq0$ for any $P_2>0$.
   By the continuity of $f$ and $\G$, we have 
   \begin{align*}
       f(\G(P_1)-P_1)\leq0,\quad\forall P_1\in\bar\dbS_+^n\quad\text{and }\quad f(P_2)\geq0,\quad\forall P_2\in\bar\dbS_+^n.
   \end{align*}
   On the other hand, by Riesz representation theorem, there exists $S\in\dbS^n\setminus\{0\}$ such that $f(X)=\lan S,X\ran$ for any $X\in\dbS^n$.
   Hence, we have
   \begin{align*}
       \begin{cases}
           \lan(\G^*-I)S,P_1\ran=\lan S,(\G-I)P_1\ran=f(\G(P_1)-P_1)\leq0,\quad\forall P_1\in\bar\dbS_+^n,\\
           \lan S,P_2\ran=f(P_2)\geq0,\quad\forall P_2\in\bar\dbS_+^n.
       \end{cases}
   \end{align*}
   By \autoref{lem:matrix} (iii), we have $\G^*(S)-S\leq0$ and $S\geq0$.
   By (ii), $S=0$, which is impossible.

   (ii) $\implies$ (iii).
   Let $H\in\bar\dbS_+^n\setminus\{0\}$ and $\G^*(H)=\lambda H$.
   If $\lambda\leq1$, then $\G^*(H)=\lambda H\leq H$.
   Hence, $H=0$ by (ii), which gives a contradiction.

   (iii) $\implies$ (ii).
   We argue by contradiction.
   Suppose that there exists $S\in\bar\dbS_+^n\setminus\{0\}$ such that $\G^*(S)\leq S$.
   Define 
   \begin{align*}
       \cC_S\triangleq\{X\in\bar\dbS_+^n:\exists c\geq0\text{ such that }X\leq cS\}.
   \end{align*}
   If $X\in\cC_S$, then there exists $c\geq0$ such that $0\leq X\leq cS$.
   Note that $\lan\G^*(P),Q\ran=\lan P,\G(Q)\ran\geq0$ for any $P,Q\in\bar\dbS_+^n$ by \autoref{lem:matrix} (iii).
   Thus, $\G^*$ is positive, which gives
   \begin{align*}
       0\leq\G^*(X)\leq c\G^*(S)\leq cS,
   \end{align*}
  Thus, $\G^*(X)\in\cC_S$ and $\G^*(\cC_S)\subset\cC_S$.
  Let $V_S:=\Span\cC_S=\cC_S-\cC_S$. Then $V_S$ is $\G^*$-invariant.
   Now we equip $V_S$ with the order-unit norm (see \citealp{paulsen2009vector}) induced by $S$:
   \begin{align*}
       \|X\|_S\triangleq\inf\{c\geq0:-cS\leq X\leq cS\},\quad\forall X\in V_S.
   \end{align*}
   Let $\bar\G^*$ be the restriction of $\G^*$ to $V_S$.
   If $-cS\leq X\leq cS$, then $-c\bar\G^*(S)\leq\bar\G^*(X)\leq c\G^*(S)$.
   By $0\leq\bar\G^*(S)\leq S$, we have $-cS\leq\bar\G^*(X)\leq cS$.
   Hence, $\|\bar\G^*(X)\|_S\leq\|X\|_S$ for any $X\in V_S$.
   Then $r(\bar\G^*)\leq1$.
   On the other hand, since $\cC_S$ is a closed, pointed, solid, and convex cone in $V_S$ and $\bar\G^*(\cC_S)\subset\cC_S$, there exists $0\neq H\in\cC_S\subset\bar\dbS_+^n$ such that $\G^*(H)=r(\bar\G^*)H$ by \autoref{lem:K-R}. 
   This contradicts (iii).
   The proof is done.
\end{proof}

We next characterize $L^2$-stabilizability of system \eqref{eq:canonical} in the special case $B=0$, which serves as a basis for the analysis of the general case.

\begin{proposition}
\label{prop:stabilizable-B=0}
    Consider the system 
    \begin{align}
        \label{eq:system-B=0}
            \begin{cases}
                x_{k+1}=Ax_k+Cz_k+w_kz_k,\\
                x_0=y\in\dbR^n.
            \end{cases}
        \end{align}
        Then the following results hold.
    \begin{enumerate}[(i)]
        \item System \eqref{eq:system-B=0} is $L^2$-stabilizable if and only if there exists $P\in\dbS_+^n$ such that 
        \begin{align}
        \label{eq:stable-B=0}
            P+CPC^{\top}-APA^{\top}>0.
        \end{align}
        \item 
        Suppose that \autoref{ass:Ainvertible} holds.
        Then system \eqref{eq:system-B=0} is $L^2$-stabilizable if and only if, for any $H\in\bar\dbS_+^n\setminus\{0\}$ satisfying $\cL^*(H)=\lambda H$, one has $\lambda>1$.
    \end{enumerate}
\end{proposition}

\begin{proof}
     (i).
      By the definition of $L^2$-stabilizability and \autoref{lem:stable}, system \eqref{eq:system-B=0} is $L^2$-stabilizable if and only if there exist $Q\in\dbS_+^n$ and $\Theta\in\dbR^{n\times n}$ such that 
     \begin{align}
     \label{eq:stabilizable-B=0}
         (A+C\Theta)^{\top}Q(A+C\Theta)+\Theta^{\top}Q\Theta-Q<0.
     \end{align}
     
     Suppose that \eqref{eq:stabilizable-B=0} holds for some $(Q,\Theta)\in\dbS_+^n\times\dbR^{n\times n}$.
    Let $\cR:=Q+C^{\top}QC$ and $\cN:=C^{\top}QA$.
    Then, by completing the square in $\Theta$, we have
    \begin{align*}
        &A^{\top}(Q-QC\cR^{-1}C^{\top}Q)A-Q\\
        \leq&(\Theta+\cR^{-1}\cN)^{\top}\cR(\Theta+\cR^{-1}\cN)-\cN^{\top}\cR^{-1}\cN+A^{\top}QA-Q\\
        =&(A+C\Theta)^{\top}Q(A+C\Theta)+\Theta^{\top}Q\Theta-Q<0
    \end{align*}
    Let $P\triangleq Q^{-1}$.
    Then $Q-QC\cR^{-1}C^{\top}Q=(P+CPC^{\top})^{-1}$.
    Hence, we have
    \begin{align*}
        A^{\top}(P+CPC^{\top})^{-1}A-P^{-1}=A^{\top}(Q-QC\cR^{-1}C^{\top}Q)A-Q<0,
    \end{align*}
   and then \eqref{eq:stable-B=0} holds by \autoref{lem:matrix} (ii).

   Conversely, suppose there exists $P\in\dbS_+^n$ such that \eqref{eq:stable-B=0} holds.
   Define $Q\triangleq P^{-1}\in\dbS_+^n$ and $\cR=Q+C^{\top}QC$.
    Then $A^{\top}(Q-QC\cR^{-1}C^{\top}Q)A-Q<0$ by the above discussion.
    Let $\Theta=-(Q+C^{\top}QC)^{-1}C^{\top}QA$.
    Then 
    \begin{align*}
        (A+C\Theta)^{\top}Q(A+C\Theta)+\Theta^{\top}Q\Theta-Q=A^{\top}(Q-QC\cR^{-1}C^{\top}Q)A-Q<0.
    \end{align*}
    It follows that \eqref{eq:stabilizable-B=0} holds and then system \eqref{eq:system-B=0} is $L^2$-stabilizable.

    (ii).
    By (i) and \eqref{eq:hA-A}, $L^2$-stabilizability of system \eqref{eq:system-B=0} is equivalent to $\cL(P)>P$ for some $P\in\dbS_+^n$.
    Then (ii) follows by \autoref{prop:positive-operator}.
    The proof is done.
\end{proof}

Before treating the general case, we show that exact null controllability implies $L^2$-stabilizability.

\begin{lemma}
\label{lem:controllable-stabilizable}
    If system \eqref{eq:canonical} is exactly null controllable, then system \eqref{eq:canonical} is $L^2$-stabilizable.
\end{lemma}

\begin{proof}
    By \autoref{thm:null-controllable}, there exists $N\in\dbN$ such that system \eqref{eq:canonical} is exactly null controllable on $[0,N]$.
    For any $y\in\dbR^n$, there exists $(u,z)\in l_{\dbF}^2(0,N;\dbR^m)\times l_{\dbF}^2(0,N;\dbR^n)$ such that $x_{N+1}=0$.
    Define $(\bar u,\bar z)\in l_{\dbF}^2(0,\infty;\dbR^m)\times l_{\dbF}^2(0,\infty;\dbR^n)$ and $\bar x\in l_{\dbF}^2(0,\infty;\dbR^n)$ by
    \begin{align*}
        \bar u_k=
        \begin{cases}
            u_k,\quad k\in[0,N],\\
            0,\quad k\geq N+1,
        \end{cases}\quad
        \bar z_k=
        \begin{cases}
            z_k,\quad k\in[0,N],\\
            0,\quad k\geq N+1,
        \end{cases}\quad
        \bar x_k=
        \begin{cases}
            x_k,\quad k\in[0,N+1],\\
            0,\quad k>N+1.
        \end{cases}
    \end{align*}
    Then $(\bar x,\bar u,\bar z)$ satisfies \eqref{eq:canonical} for any $k\in\dbN$.
    By \autoref{lem:stabilizable}, system \eqref{eq:canonical} is $L^2$-stabilizable.
    The proof is done.
\end{proof}

Combining \autoref{lem:controllable-stabilizable} with the controllability decomposition in \autoref{thm:controllability decomposition}, we obtain the following reduction of $L^2$-stabilizability of system \eqref{eq:canonical}.

\begin{theorem}[Stabilizability decomposition]
\label{thm:stabilizable-decomposition}
    Let \autoref{ass:Ainvertible} hold.
    Suppose that $0<k\triangleq\dim V_{n-1}<n$, where $V_{n-1}$ is defined by \eqref{eq:V}.
    Let $P\in\dbR^{n\times n}$ be the orthogonal matrix satisfying \eqref{eq:decomposition} and \eqref{eq:decompostion-original-matrix}.
    The system \eqref{eq:canonical} is $L^2$-stabilizable if and only if the subsystem 
    \begin{align}
    \label{eq:subsystem-2}
        \begin{cases}
            y_{2,k+1}=A_{22}y_{2,k}+C_{22}v_{2,k}+w_kv_{2,k},\\ y_{2,0}\in\dbR^{n-k},
        \end{cases}
    \end{align}
    is $L^2$-stabilizable, where $A_{22},C_{22}$ are defined by \eqref{eq:A11-hA11}.
\end{theorem}

\begin{proof}
    Let 
    \begin{align*}
        y_k=
        \begin{pmatrix}
            y_{1,k}\\y_{2,k}
        \end{pmatrix}
        \triangleq P^{\top}x_k,\quad
        v_k=
        \begin{pmatrix}
            v_{1,k}\\v_{2,k}
        \end{pmatrix}
        \triangleq P^{\top}z_k,\quad k\in\dbN,
    \end{align*}
    where $y_{1,k},v_{1,k}\in\dbR^k$ and $y_{2,k},v_{2,k}\in\dbR^{n-k}$.
    Then, by \eqref{eq:decompostion-original-matrix} and \eqref{eq:canonical}, we have
    \begin{align}
    \label{eq:state-decomposition}
        \begin{cases}
            y_{1,k+1}=A_{11}y_{1,k}+A_{12}y_{2,k}+B_1u_k+C_{11}v_{1,k}+C_{12}v_{2,k}+w_kv_{1,k},\\
            y_{2,k+1}=A_{22}y_{2,k}+C_{22}v_{2,k}+w_kv_{2,k}.
        \end{cases}
    \end{align}
    Hence, $\left(\begin{matrix}
        F\\K
    \end{matrix}\right)$ is a stabilizer of system \eqref{eq:canonical} if and only if $\left(\begin{matrix}
        FP\\P^{\top}KP
    \end{matrix}\right)$ is a stabilizer of system \eqref{eq:state-decomposition}, which implies that $L^2$-stabilizability of system \eqref{eq:canonical} is equivalent to $L^2$-stabilizability of system \eqref{eq:state-decomposition}.

    Suppose that system \eqref{eq:state-decomposition} is $L^2$-stabilizable.
    Then its subsystem \eqref{eq:subsystem-2} is obviously $L^2$-stabilizable.
    Conversely, suppose that system \eqref{eq:subsystem-2} is $L^2$-stabilizable.
    Then there exist $K_2\in\dbR^{(n-k)\times(n-k)}$ and $P_{22}\in\dbS_+^{n-k}$ such that 
    \begin{align*}
        \Delta_{22}\triangleq P_{22}-M_{22}^{\top}P_{22}M_{22}-K_2^{\top}P_{22}K_2>0,
    \end{align*}
    where $M_{22}=A_{22}+C_{22}K_2$.
    By \autoref{thm:controllability decomposition} and \autoref{lem:controllable-stabilizable}, system \eqref{eq:subsystem-1} is $L^2$-stabilizable.
    Thus, there exist $F_1\in\dbR^{m\times k}$, $K_1\in\dbR^{k\times k}$ and $P_{11}\in\dbS_+^k$ such that 
    \begin{align*}
        \Delta_{11}\triangleq P_{11}-M_{11}^{\top}P_{11}M_{11}-K_1^{\top}P_{11}K_1>0,
    \end{align*}
    where $M_{11}=A_{11}+B_1F_1+C_{11}K_1$.
    Plugging $u=F_1y_1$, $v_1=K_1y_1$ and $v_2=K_2y_2$ into system \eqref{eq:state-decomposition}, we have 
    \begin{align*}
        y_{k+1}=\tA y_k+w_k\tC y_k,
    \end{align*}
    where 
    \begin{align*}
        \tA=
        \begin{pmatrix}
            M_{11}&M_{12}\\
            0&M_{22}
        \end{pmatrix},\quad
        \tC=
        \begin{pmatrix}
            K_1&0\\
            0&K_2
        \end{pmatrix},\quad M_{12}=A_{12}+C_{12}K_2.
    \end{align*}
    Let $P_{\epsilon}=\diag(\epsilon P_{11},P_{22})$, where $\epsilon>0$.
    Then $P_{\epsilon}\in\dbS_+^n$ for any $\epsilon>0$.
    Note that 
    \begin{align*}
        \Delta_{\epsilon}\triangleq P_{\epsilon}-\tA^{\top}P_{\epsilon}\tA-\tC^{\top}P_{\epsilon}\tC=
        \begin{pmatrix}
            \epsilon\Delta_{11}&-\epsilon M_{11}^{\top}P_{11}M_{12}\\-
            \epsilon M_{12}^{\top}P_{11}M_{11}&\Delta_{22}-\epsilon M_{12}^{\top}P_{11}M_{12}
        \end{pmatrix}.
    \end{align*}
    In the above block matrix, the Schur complement of $\epsilon\Delta_{11}$ is
    \begin{align*}
        \Delta_{22}-\epsilon M_{12}^{\top}P_{11}M_{12}-\epsilon M_{12}^{\top}P_{11}M_{11}\Delta_{11}^{-1}M_{11}^{\top}P_{11}M_{12}.
    \end{align*}
    Since $\Delta_{22}>0$, the above Schur complement is positive definite for sufficiently small $\epsilon>0$.
    Hence, $\Delta_{\epsilon}>0$.
    Therefore, system \eqref{eq:state-decomposition} is $L^2$-stabilizable.
    The proof is done.
\end{proof}

Combining the stabilizability decomposition in \autoref{thm:stabilizable-decomposition} with the spectral characterization in \autoref{prop:stabilizable-B=0}, we obtain the following Hautus-type criterion for system \eqref{eq:canonical}.

\begin{theorem}
\label{thm:stabilizable-hautus}
    Under \autoref{ass:Ainvertible}, the following statements are equivalent.
    \begin{enumerate}[(i)]
        \item System \eqref{eq:canonical} is $L^2$-stabilizable.
        \item For any eigenpair $(\lambda,H)\in[0,1]\times\bar\dbS_+^n\setminus\{0\}$ of $\cL^*$, we have $\hB^{\top}H\neq0$.
    \end{enumerate}
\end{theorem}

\begin{proof}
    (i) $\implies$ (ii).
    Let $V_{n-1}$ be defined by \eqref{eq:V}, and set $k\triangleq\dim V_{n-1}$.
    If $k=n$, then system \eqref{eq:canonical} is exactly null controllable, and (ii) follows from \autoref{lem:controllable-stabilizable} and \autoref{thm:controllable-hautus}.
    If $k=0$, then $B=0$, and (ii) follows from \autoref{prop:stabilizable-B=0}.
    Now assume $0<k<n$, and let $P$ be the orthogonal matrix in \autoref{thm:controllability decomposition} satisfying \eqref{eq:decomposition}.
    Let $(\lambda,H)\in[0,1]\times\bar\dbS_+^n\setminus\{0\}$ satisfy $\cL^*(H)=\lambda H$.
    Write $\bar H:=P^{\top}HP$ in block form
    \begin{align*}
        \bar H=
        \begin{pmatrix}
            \bar H_{11}&\bar H_{12}\\
            \bar H_{12}^{\top}&\bar H_{22}
        \end{pmatrix},\quad
        \bar H_{11}\in\bar\dbS_+^k,\quad\bar H_{22}\in\bar\dbS_+^{n-k},\quad\bar H_{12}\in\dbR^{k\times(n-k)}.
    \end{align*}
    We now show that $\hB^{\top}H\neq0$ by contradiction.
    Suppose that $\hB^{\top}H=0$.
    Then 
    \begin{align*}
        0=\hB^{\top}HP=\hB^{\top}PP^{\top}HP=(\hB_1^{\top},0)\bar H=(\hB_1^{\top}\bar H_{11},\hB_1^{\top}\bar H_{12}).
    \end{align*}
    Hence, $\hB_1^{\top}\bar H_{11}=0$.
    By $\cL^*(H)=\lambda H$, we have $\lambda P^{\top}HP=P^{\top}\cL^*(H)P$, which implies that
    \begin{align*}
        \lambda
        \begin{pmatrix}
            \bar H_{11}&\bar H_{12}\\
            \bar H_{12}^{\top}&\bar H_{22}
        \end{pmatrix}=P^{\top}\hA^{\top}PP^{\top}HPP^{\top}\hA P+P^{\top}\hC^{\top}PP^{\top}HPP^{\top}\hC P=
        \begin{pmatrix}
            \Lambda_{11}&\Lambda_{12}\\
            \Lambda_{12}^{\top}&\Lambda_{22}
        \end{pmatrix},
    \end{align*}
    where 
    \begin{align*}
        \begin{cases}
            \Lambda_{11}=\hA_{11}^{\top}\bar H_{11}\hA_{11}+\hC_{11}^{\top}\bar H_{11}\hC_{11},\\
            \Lambda_{12}=\hA_{11}^{\top}(\bar H_{11}\hA_{12}+\bar H_{12}\hA_{22})+\hC_{11}^{\top}(\bar H_{11}\hC_{12}+\bar H_{12}\hC_{22}),\\
            \Lambda_{22}=\hA_{12}^{\top}(\bar H_{11}\hA_{12}+\bar H_{12}\hA_{22})+\hA_{22}^{\top}(\bar H_{12}^{\top}\hA_{12}+\bar H_{22}\hA_{22})\\
            \qquad\quad+\hC_{12}^{\top}(\bar H_{11}\hC_{12}+\bar H_{12}\hC_{22})+\hC_{22}^{\top}(\bar H_{12}^{\top}\hC_{12}+\bar H_{22}\hC_{22}).
        \end{cases}
    \end{align*}
    Then, we have 
    \begin{align*}
        \hA_{11}^{\top}\bar H_{11}\hA_{11}+\hC_{11}^{\top}\bar H_{11}\hC_{11}=\lambda\bar H_{11},\quad\hB_1^{\top}\bar H_{11}=0.
    \end{align*}
    By \autoref{thm:controllability decomposition}, system \eqref{eq:subsystem-1} is exactly null controllable.
    By \autoref{thm:controllable-hautus}, $\bar H_{11}=0$.
    Since $\bar H\geq0$, we have $\bar H_{12}=0$.
    By $\Lambda_{22}=\lambda\bar H_{22}$, we have 
    \begin{align*}
        \hA_{22}^{\top}\bar H_{22}\hA_{22}+\hC_{22}^{\top}\bar H_{22}\hC_{22}=\lambda\bar H_{22}.
    \end{align*}
    By \autoref{thm:stabilizable-decomposition}, system \eqref{eq:subsystem-2} is $L^2$-stabilizable.
    By \autoref{prop:stabilizable-B=0}, $\bar H_{22}=0$.
    Hence, $\bar H=0$, which is a contradiction.

    (ii) $\implies$ (i).
    If $\dim V_{n-1}=n$, then system \eqref{eq:canonical} is exactly null controllable by \autoref{thm:null-controllable}.
    By \autoref{lem:controllable-stabilizable}, system \eqref{eq:canonical} is $L^2$-stabilizable.
    If $\dim V_{n-1}=0$, then $\hB=0$.
    Let $H\in\bar\dbS_+^n\setminus\{0\}$ such that $\cL^*(H)=\lambda H$.
    By (ii), we have $\lambda>1$.
    Then system \eqref{eq:canonical} is $L^2$-stabilizable by \autoref{prop:stabilizable-B=0}.
    Now assume $0<k<n$, and let $P$ be the orthogonal matrix in \autoref{thm:controllability decomposition} satisfying \eqref{eq:decomposition} and \eqref{eq:decompostion-original-matrix}.
    We argue by contradiction.
    Suppose that system \eqref{eq:canonical} is not $L^2$-stabilizable.
    By \autoref{thm:stabilizable-decomposition}, the subsystem \eqref{eq:subsystem-2} is not $L^2$-stabilizable.
    By \autoref{prop:stabilizable-B=0}, there exists $(\lambda,H_{22})\in[0,1]\times\bar\dbS_+^{n-k}\setminus\{0\}$ such that $\hA_{22}^{\top}H_{22}\hA_{22}+\hC_{22}^{\top}H_{22}\hC_{22}=\lambda H_{22}$.
    Define $H\triangleq P\diag(0,H_{22})P^{\top}\in\bar\dbS_+^n\setminus\{0\}$.
    Then 
    \begin{align*}
        P^{\top}\cL^*(H)P=\diag(0,\hA_{22}^{\top}H_{22}\hA_{22}+\hC_{22}^{\top}H_{22}\hC_{22})=\lambda\diag(0,H_{22}).
    \end{align*}
    Hence, $\cL^*(H)=\lambda H$.
    On the other hand, 
    \begin{align*}
        \hB^{\top}H=\hB^{\top}PP^{\top}HPP^{\top}=(\hB_1^{\top},0)\diag(0,H_{22})P^{\top}=0,
    \end{align*}
    which contradicts (ii).
    The proof is done.
\end{proof}

\section{Conclusion}
\label{sec:conclusion}

This paper has investigated exact controllability and $L^2$-stabilizability of discrete-time backward-structured stochastic linear systems. 
For each prescribed finite horizon, equivalent characterizations of exact controllability have been established in terms of the controllability Gramian, a Kalman-type rank condition, and the reachable subspace. 
Without prescribing a control horizon, exact null controllability has been characterized by a finite-dimensional rank condition and a Hautus-type criterion based on the spectral properties of positive operators. 
For $L^2$-stabilizability, a stabilizability decomposition has been established, leading to a corresponding Hautus-type spectral criterion.
In addition, exact null controllability has been shown to imply $L^2$-stabilizability. 
These results extend classical algebraic and spectral tests to discrete-time backward-structured stochastic linear systems.

\section*{Disclosure statement}

No potential conflict of interest was reported by the author.

\section*{Funding}

This research received no specific grant from any funding agency in the public, commercial, or not-for-profit sectors.

\section*{Data availability statement}

Data sharing is not applicable to this article as no new data were created or analyzed in this study.


\end{document}